\pdfoutput=1
\RequirePackage{ifpdf}
\ifpdf 
\documentclass[pdftex]{sigma}
\else
\documentclass{sigma}
\fi

\usepackage{mathtools}
\usepackage{tikz-cd}
\usepackage{difftrees}

\usepackage{amscd}

\newcommand{\tr}{\triangleright}
\newcommand{\vartr}{\blacktriangleright}

\newcommand{\OT}{\operatorname{OT}}

\newcommand{\dpr}{\! \cdot \!}  

\newcommand{\T}[1][]{\ensuremath{\mathrm{T}#1}}       

\newcommand{\s}{\ensuremath{\mathfrak{s}}}

\newcommand{\ZZ}{\ensuremath{\mathbb{Z}}}

\usepackage{tikz-cd}

\newcommand*{\preLie}{\operatorname{preLie}}

\newcommand{\Der}{\operatorname{Der}}
\newcommand{\Img}{\operatorname{Im}}

 \usepackage{stmaryrd}
 \newcommand{\refp}[1]{(\ref{#1})}

    \newcommand*{\LTS}{\operatorname{LTS}}
        \newcommand*{\LAT}{\operatorname{LAT}}
      \newcommand*{\Lie}{\operatorname{Lie}}
\newcommand*{\Mag}{\operatorname{Mag}}
\newcommand*{\Alg}{\operatorname{Alg}}
\newcommand{\spn}{\operatorname{span}}
\newcommand{\ord}[1]{\overrightarrow{#1}}
\renewcommand*{\OT}{\ord{\operatorname{Tree}}}
\newcommand*{\Tree}{{\operatorname{Tree}}}
\renewcommand*{\T}{{\mathcal{T}}}
\newcommand{\J}{\mathcal{J}}
\newcommand{\I}{\mathcal{I}}
\newcommand{\noarg}{\,\_\,}

\numberwithin{equation}{section}

\newtheorem{Theorem}{Theorem}[section]
\newtheorem*{Theorem*}{Theorem}
\newtheorem{Corollary}[Theorem]{Corollary}
\newtheorem{Lemma}[Theorem]{Lemma}
\newtheorem{Proposition}[Theorem]{Proposition}
 { \theoremstyle{definition}
\newtheorem{Definition}[Theorem]{Definition}
\newtheorem{Note}[Theorem]{Note}
\newtheorem{Example}[Theorem]{Example}
\newtheorem{Remark}[Theorem]{Remark} }

\begin{document}

\allowdisplaybreaks

\renewcommand{\thefootnote}{}

\newcommand{\arXivNumber}{2306.15582}

\renewcommand{\PaperNumber}{068}

\FirstPageHeading

\ShortArticleName{Lie Admissible Triple Algebras: The Connection Algebra of Symmetric Spaces}

\ArticleName{Lie Admissible Triple Algebras:\\ The Connection Algebra of Symmetric Spaces\footnote{This paper is a~contribution to the Special Issue on Symmetry, Invariants, and their Applications in honor of Peter J.~Olver. The~full collection is available at \href{https://www.emis.de/journals/SIGMA/Olver.html}{https://www.emis.de/journals/SIGMA/Olver.html}}}

\Author{Hans Z. MUNTHE-KAAS~$^{\rm a}$ and Jonatan STAVA~$^{\rm b}$}

\AuthorNameForHeading{H.Z.~Munthe-Kaas and J.~Stava}

\Address{$^{\rm a)}$~Department of Mathematics and Statistics, UiT The Arctic University of Norway,\\
\hphantom{$^{\rm a)}$}~P.O.~Box 6050, Stakkevollan, 9037 Troms{\o}, Norway}
\EmailD{\href{mailto:hans.munthe-kaas@uit.no}{hans.munthe-kaas@uit.no}}
\URLaddressD{\url{http://hans.munthe-kaas.no}}

\Address{$^{\rm b)}$~Department of Mathematics, University of Bergen, P.O.~Box 7803, 5020~Bergen, Norway}
\EmailD{\href{mailto:jonatan.stava@uib.no}{jonatan.stava@uib.no}}

\ArticleDates{Received December 20, 2023, in final form July 03, 2024; Published online July 25, 2024}

\Abstract{Associated to a symmetric space there is a canonical connection with zero torsion and parallel curvature. This connection acts as a binary operator on the vector space of smooth sections of the tangent bundle, and it is linear with respect to the real numbers. Thus the smooth section of the tangent bundle together with the connection form an algebra we call the \emph{connection algebra}. The constraints of zero torsion and constant curvature makes the connection algebra into a Lie admissible triple algebra. This is a type of algebra that generalises pre-Lie algebras, and it can be embedded into a post-Lie algebra in a canonical way that generalises the canonical embedding of Lie triple systems into Lie algebras. The free Lie admissible triple algebra can be described by incorporating triple-brackets into the leaves of rooted (non-planar) trees.}

\Keywords{Lie admissible triple algebra; connection algebra; symmetric spaces}

\Classification{53C05; 17D25; 05C05; 53C35}

\begin{flushright}
\begin{minipage}{55mm}
\it Dedicated to Peter Olver\\
 in the honour of his 70th birthday
 \end{minipage}
 \end{flushright}

\renewcommand{\thefootnote}{\arabic{footnote}}
\setcounter{footnote}{0}

\section{Introduction}

The connection algebra of a smooth manifold with an affine connection is the vector space of smooth sections of the tangent bundle with the connection viewed as an $\mathbb{R}$-linear binary operator. On a Euclidean space this algebra is a pre-Lie algebra, and on a Lie group it is a~post-Lie algebra. Understanding the connection algebra is interesting by itself, and it also has potential applications for developing algorithms for numerical integration. In the case of pre-Lie algebras,~\cite{B-series_CHV_2010} and \cite{CALAQUE_2011} have demonstrated its relation to numerical integration and B-series. Methods for numerical integration on Lie groups developed in \cite{MK_1995_LB} and \cite{iserles2000lie} laid the foundation for LB-series, a generalization of B-series for which the underlying algebraic structure is a post-Lie algebra~\cite{CEFMK19, KMKL_15}.

Pre-Lie algebras were first defined by Gerstenhaber \cite{Gerstenhaber_63} and Vinberg \cite{Vinberg}. Chapoton and Livernet proved in \cite{Chapoton_and_Livernet_01} that the free pre-Lie algebra is given as the span of non-planar rooted trees. Another proof of this was given by L\"{o}fwall and Dzhumadil'daev in \cite{Dzhumadildeav_and_Lofwall_02}. Post-Lie algebras were first described by Vallette in \cite{BV_2007}, but this structure also appears in the D-algebra described in \cite{MK_Wright_08}. In \cite{MK_Lundervold_2013, BV_2007}, it is also shown that the free post-Lie algebra is given as the free Lie algebra over the planar rooted trees. An explicit basis and the dimension of the graded components are given in \cite{MKK_03}.

Symmetric spaces are the natural next level of complexity beyond Lie groups. Due to Nomizu~\cite{Nomizu_54}, there is a canonical connection associated to a symmetric space with zero torsion and parallel curvature. In this article we will focus on the type of algebra defined by this connection, a type of algebra we have named Lie admissible triple (LAT) algebra. Geometric integration algorithms have recently been introduced for symmetric spaces~\cite{munthekaas24gio}. The analysis of numerical algorithms on symmetric spaces motivates the study of (Lie)--Butcher series on symmetric spaces. In the present work, we characterise the free algebra generated by the canonical connection on symmetric spaces. This is a first step in developing a comprehensive theory for series developments on symmetric spaces.

When investigating connection algebras, it is convenient to define a triple-bracket as the skew-symmetrization of the associator. Geometrically this is the Ricci identity \cite{Ricci_1900}. A LAT algebra is defined by two relations, both best described using the triple-bracket.

A key question when investigating an algebraic structure is what the free object looks like. The main result of this article provides a basis for the free LAT algebra. This is presented in Theorem \ref{th:LATbasis}. To be able to describe the free LAT algebra it is necessary to have a description of the free algebra, which is the linear span of the free magma. The free algebra is commonly described by planar rooted trees, but when working with the triple-bracket this description is inconvenient. We present an alternative basis for the free algebra that works well with the triple-bracket. Theorem \ref{th:OSBB} provides the theoretical link between the planar rooted trees and this new basis. This theorem is purely combinatorical and might be interesting beyond the context of this article.

The article is organised as follows: Section \ref{sec2} introduces the concept of connection algebras and reviews some of the known results in this area. The link between symmetric spaces and LAT algebras will also be made clear. In Section \ref{sec3}, we discuss tensor-algebras, D-algebras and the free algebra, the goal being to state the main result in Theorem \ref{th:LATbasis}. The proof of this theorem will be given later in the paper, but the tools developed in this section is already enough to give a simple proof of the known fact that the free pre-Lie algebra can be described by non-planar rooted trees, see Theorem \ref{th:pre-Lie}.

While the main goal of the next four sections are to provide a proof of the main theorem, these sections do contain some results that are interesting on their own. In Section \ref{sec:Hall}, in Proposition~\ref{prop:LTS_Hall}, we provide a Hall basis for Lie triple systems. In Section \ref{sec:OSBB}, we present an alternative basis for the free tensor algebra. This basis consists of elements that combine symmetric and skew-symmetric parts under a few simple conditions that ensure uniqueness, see Theorem \ref{th:OSBB}. Section \ref{sec:bases} utilizes this theorem to provide a basis for the free algebra that is convenient when working with the LAT relations, before we finally prove the main theorem in Section \ref{sec:proof}.

In the last section, Section \ref{sec:PL-emb}, we relate LAT algebras to post-Lie algebras by an embedding that generalize the standard embedding of a Lie triple system into a $\mathbb{Z}_2$-graded Lie algebra provided by Jacobson in 1951 \cite{Jacobson1951}. This result is presented in Theorem \ref{th:pL_emb}.

\section{The connection algebra}\label{sec2}

Consider a (real) smooth manifold $M$ and denote by $\mathfrak{X}_M= \Gamma(TM)$ the vector space of smooth vector fields on $M$. Let $\nabla$ be an affine connection on $M$. Then $(\mathfrak{X}_M, \nabla)$ will be an algebra over~$\mathbb{R}$ which we will call the \textit{connection algebra} of $M$. Two important geometric objects related to the connection are the curvature
\begin{equation*}
 R(X,Y)Z = \nabla_X \nabla_Y Z - \nabla_Y \nabla_X Z - \nabla_{[X,Y]_J} Z,
\end{equation*}
and the torsion
\begin{equation*}
 T(X,Y) = \nabla_X Y - \nabla_Y X - [X,Y]_J,
\end{equation*}
where $[\noarg , \noarg ]_J$ is the Jacobi--Lie bracket of vector fields. Since we will work with the connection as a product in a non-associative algebra it will be convenient to introduce an alternative notation. We let $X \rhd Y := \nabla_X Y$ whenever the latter is defined, but we will also use notation $\rhd$ for a binary product in a purely algebraic setting. The (negative) associator of $\rhd$ is given by $\text{ass}(x,y,z)= x \rhd (y \rhd z) - (x \rhd y ) \rhd z$ and we define the triple-bracket as
\begin{equation*}
 [x,y,z] := \text{ass}(x,y,z) - \text{ass}(y,x,z),
\end{equation*}
and we immediately note that
\begin{equation*}
 R(X,Y)Z - T(X,Y) \rhd Z = [X,Y,Z].
\end{equation*}
This relation is called the Ricci identity, although it is often stated in the torsion free case \cite{Besse_1987, Ricci_1900}.

The Bianchi identities are two fundamental relations connecting the torsion and the Rimannian curvature tensor of a manifold. For $X,Y,Z \in \mathfrak{X}_M$, we have
\begin{alignat*}{3}
& 0= \sum_{\circlearrowleft (X,Y,Z)} T(T(X,Y),Z) + (\nabla_X T)(Y,Z) - R(X,Y)Z , \qquad && \mbox{(1 Bianchi)},&\\
&0= \sum_{\circlearrowleft (X,Y,Z)} (\nabla_X R)(Y,Z) - R(X,T(Y,Z)) , && \mbox{(2 Bianchi)},& 
\end{alignat*}

where $\sum_{\circlearrowleft}$ means the sum over all cyclic permutation of the indicated arguments.

\subsection{Some algebras related to constant curvature and torsion}

Various geometries are related to properties of the triple-bracket of the connection:
\begin{itemize}\itemsep=0pt
 \item A \emph{pre-Lie algebra} is $(\mathcal{A},\rhd)$ where $[x,y,z] \equiv 0$. A manifold has a pre-Lie connection on~$\mathfrak{X}_M$ if and only if $M$ is locally an Abelian Lie group, i.e., locally a Euclidean space.
 \item A \emph{Lie admissible algebra} is $(\mathcal{A},\rhd)$ where $[x,y,z]+[y,z,x] + [z,x,y]\equiv 0$. The bracket defined by the skew product $[x,y] := x\rhd y - y\rhd x$ is a \emph{Lie bracket} if and only if $(\mathcal{A},\rhd)$ is Lie admissible. A connection algebra is Lie admissible if and only if the connection is \emph{torsion free}, $T=0$, in which case $X\rhd Y -Y\rhd X = [X,Y]_J$ (see \cite[Proposition 2.13]{MK_Stern_Verdier_20}). This carries little geometric information, since any connection has a related torsion free connection with the same geodesics.
 \item A \emph{post-Lie algebra} is $(\mathcal{A},\rhd,[\noarg,\noarg])$ where $[\noarg,\noarg]$ is a Lie bracket and where $\rhd$ and $[\noarg,\noarg]$ relate as
 \begin{align*}
 & [x,y]\rhd z= [x,y,z],\qquad x\rhd [y,z]= [x\rhd y,z]+ [y,x\rhd z].
 \end{align*}
 A connection algebra $(\mathfrak{X}_M,\nabla)$ is post-Lie if $R\equiv 0$ and $\nabla T\equiv 0$, in which case there are two Lie brackets on the vector fields $[X,Y] = -T(X,Y)$ and $[X,Y]_J = X\rhd Y-Y\rhd X + [X,Y]$. A post-Lie connection exists on $\mathfrak{X}_M$ if and only if $M$ is locally a Lie group. More generally, a post-Lie structure exists on an anchored vector bundle if and only if the bundle is an action Lie algebroid~\cite{MK_Lundervold_2013}.
\end{itemize}

In this paper, we are interested in the canonical connection on a locally symmetric space (such as the Levi-Civita connection on a Riemannian symmetric space). This is a connection where $T \equiv 0$ and $\nabla R\equiv 0$.

\begin{Remark}
For any locally symmetric space, there is a canonical connection such that $\nabla R=0$ and $T=0$. Equivalently, any manifold with such a connection is a locally symmetric space. A~locally symmetric space is a (globally) symmetric space if it is connected, simply connected and geodesically complete \cite{Loos1, Nomizu_54}. 
\end{Remark}

\begin{Definition}[LAT] An algebra $(\mathcal{A},\rhd)$ is called a \emph{Lie admissible triple algebra}
if the triple bracket verifies
\begin{align}
&[x,y,z]+[y,z,x] + [z,x,y]\equiv 0,\label{eq:LAT1}\\
&w\rhd [x,y,z] - [w\rhd x,y,z] - [x,w\rhd y, z] - [x,y,w\rhd z]\equiv 0.\label{eq:LAT2}
\end{align}
\end{Definition}

A connection algebra with $T=0$ verifies~\refp{eq:LAT1}, hence $[X,Y,Z] = R(X,Y)Z$ and $\nabla R=0$
is expressed in~\refp{eq:LAT2}. So the algebra of a torsion free connection with constant curvature is a Lie admissible triple algebra. On the other hand, if the connection algebra is LAT, then equation~(\ref{eq:LAT1}) implies that the connection is torsion free and then equation~(\ref{eq:LAT2}) gives exactly constant curvature. The proposition below follows.

\begin{Proposition}
 A connection algebra $(\mathfrak{X}_M, \rhd)$ is LAT if and only if the manifold $(M,\nabla)$ is a locally symmetric space.
\end{Proposition}

\begin{Lemma}[LTS]\label{lemma:LTS} If $(\mathcal{A},\rhd)$ is an LAT then $(\mathcal{A},[\noarg,\noarg,\noarg])$ is a \emph{Lie triple system} $($LTS$)$, defined as a vector space with a tri-linear bracket satisfying
\begin{align}
&[x,y,z]= -[y,x,z],\label{eq:LTS1}\\
&0= [x,y,z] + [y,z,x] + [z,x,y] ,\label{eq:LTS2}\\
&[u,v,[x,y,z]]= [[u,v,x],y,z]+[x,[u,v,y],z] + [x,y,[u,v,z]] .\label{eq:LTS3}
\end{align}
\end{Lemma}

\begin{proof} Equation \refp{eq:LTS1} follows from the definition of the triple bracket, \refp{eq:LTS2} is~\refp{eq:LAT1}, while~\refp{eq:LTS3} follows from~\refp{eq:LAT2} by a tedious but straightforward computation.
\end{proof}

Lie triple systems have been studied extensively in relation to symmetric spaces, see, for instance,~\cite{Loos1} and~\cite{Bertram_00}. Just as the tangent space of a Lie group has the structure of a Lie algebra, Lie triple systems appear as the tangent space of symmetric spaces. In recent years, the theory of post-Lie algebras has been developed in different directions. Post-Lie algebras form a refinement of Lie algebras, enriching them with structures relating to connections and geodesics. In this paper, we study the new concept of Lie admissible triple algebras, which in many ways relate to Lie triple systems similarly to the relationship between post-Lie algebras and Lie algebras.

\begin{Note}
 The algebraic structure we call Lie admissible triple has previously appeared in the context of integrable systems under the name G-algebra, although with different sign on the associator \cite{GS_00,Sokolov_17}.
\end{Note}

\section{Tensors, trees and the free Lie admissible triple algebra}\label{sec3}

The free Lie admissible triple algebra $\LAT(\mathcal{C})$ over the set $\mathcal{C}$ is defined by the universal property: If $\mathcal{A}$ is a Lie admissible triple algebra and $f \colon \mathcal{C} \rightarrow \mathcal{A}$, then there exists a unique algebra homomorphism $\phi$ such that the diagram below commutes:
\[
\begin{tikzcd}[column sep=small]
\mathcal{C} \arrow[hookrightarrow]{rr} \arrow{rrd}{f}
 && \LAT(\mathcal{C}) \arrow{d}{\phi} \\
 && \mathcal{A}.
\end{tikzcd}
\]

The main result of this article is stated in Theorem \ref{th:LATbasis} at the end of this section. In this theorem, we provide a basis for $\LAT(\mathcal{C})$ in terms of rooted trees and triple-brackets.

\subsection{Tensor algebra}\label{sec:tensor}

Consider a vector space $W$ over a field $K$ with a basis $\mathcal{C}$. Let $T(W)$ be the free tensor algebra over $W$, that is
\[
T(W) = \bigoplus_{i=0}^{\infty} W^{\otimes i} = K \oplus W \oplus (W \otimes W) \oplus (W \otimes W \otimes W) \oplus \cdots .
\]
The natural basis for $T(W)$ is given by all elements on the form
\[
x_1 \otimes x_2 \otimes \dots \otimes x_k, \qquad x_i \in \mathcal{C}, \quad k\in \mathbb{N}.
\]
For simpler notation, we will denote the tensor product between elements by a simple dot: $x \cdot y := x \otimes y$. Since the tensor product is associative but not commutative, the basis elements of~$T(W)$ can be thought of as all the \emph{words} over the \emph{alphabet} $\mathcal{C}$, and the number of basis elements in the graded component $W^{\otimes k}$ is given by the number of words of length $k$ one can make from the elements in $\mathcal{C}$ which is $n^k$, where $n:= |\mathcal{C}|$. If $\omega$ is a word in $T(W)$, we will write $|\omega|=k$ meaning that the length of the word $\omega$ is $k$. We extend this definition to any element of homogeneous degree in $T(W)$ with respect to the tensor product; if $\omega \in W^{\otimes k}$, then $|\omega|=k$, even if $\omega$ is not a word, but a linear combination of such.

Since the tensor algebra is associative, we may define a Lie bracket by
\begin{equation}\label{eq:tensor_bracket}
 [x,y] = x\cdot y - y \cdot x
\end{equation}
for $x,y \in T(W)$. In addition, we may define a map $\mathfrak{s}$ called a \emph{symmetrization map} by
\begin{equation*}
 \mathfrak{s}(x_1 \cdots x_k) = \frac{1}{k!} \sum_{\sigma \in S_k} \sigma(x_1 \cdots x_k),
\end{equation*}
where $x_1, \dots, x_k \in \mathcal{C}$ and $S_k$ is the symmetric group acting by permuting the elements of the words of length $k$. Let $\Tilde{\mathcal{J}}$ be the ideal in $T(W)$ generated by the Lie bracket $[\noarg , \noarg]$. An important relation between permutations of words and the bracket is given in the following proposition.

\begin{Proposition}\label{prop:perm}
 For any $\omega \in T(W)$ with $|\omega| = k$ and any permutation $\rho \in S_k$, we have
 \[
 \omega - \rho(\omega) \in \Tilde{\mathcal{J}}.
 \]
\end{Proposition}

\begin{proof}
 First notice that for $a,b \in \mathcal{C}$ and a word $\omega \in T(W)$, we have
 \[
 a\cdot b \cdot \omega - b \cdot a \cdot \omega = [a,b] \cdot \omega \in \Tilde{\mathcal{J}}.
 \]
 This implies that the statement is true for words of length $2$. Assume true for $|\omega| < k$ and consider
 \[
 a\cdot \omega - a \cdot \rho(\omega) = a \cdot (\omega - \rho(\omega) ) \in \Tilde{\mathcal{J}}.
 \]
 Any permutation $\sigma$ can be given as a composition of permutations that leave the first position fixed and the permutation that switch the first two positions and leave the rest fixed, hence we get
 \[
 a\cdot \omega - \sigma(a \cdot \omega) \in \Tilde{\mathcal{J}}
 \]
 and the proposition follows by induction.
\end{proof}

\subsection{D-algebra}
The algebra of covariant derivations from a general connection has the structure of an enveloping algebra of a post-Lie algebra, called a D-algebra~\cite{MK_Wright_08}. We need this basic structure in order to understand the free LAT.

For an algebra $(\mathcal{A},\rhd)$, let $(T(\mathcal{A}),\cdot)$ denote the tensor algebra. Let $D(\mathcal{A}) = (T(\mathcal{A}),\cdot,\rhd)$ be defined by extending $\rhd$ to $T(\mathcal{A})$ as
\begin{align}
&x\rhd ({u\dpr v})= {(x\rhd u)\cdot v} + {u\dpr (x\rhd v)}, \qquad \mbox{(Leibniz rule)}\label{eq:DA1}\\
&({x\dpr u})\rhd v= x\rhd(u\rhd v) - (x\rhd u) \rhd v = \text{ass}(x,u,v), \label{eq:DA2}
\end{align}
for all $x\in \mathcal{A}$ and $u,v\in D(\mathcal{A})$. 

\begin{Note} $D(\mathcal{A})$ is introduced in~\cite{MK_Wright_08} with the name \emph{D-algebra}. It forms the enveloping algebra of
the post-Lie algebra $(\mathfrak{g},\rhd,[\noarg,\noarg])$, where $\mathfrak{g} = \Lie(\mathcal{A})\subset D(\mathcal{A})$ and $[ \noarg , \noarg]$ is defined as in equation \eqref{eq:tensor_bracket}. The rich algebraic structure of the D-algebra is studied in~\cite{KMKL_15,MK_Lundervold_2013,MK_Wright_08}.
\end{Note}

\subsection{Trees and the free algebra}\label{sec:Trees}

Let $(\Alg(\mathcal{C}),\rhd)$ denote the free algebra (non-associative, non-commutative) over a set of generators $\mathcal{C}$,
also referred to as a set of \emph{colours}, or sometimes an \emph{alphabet}.
The free magma $\Mag(\mathcal{C})$ is the set of all bracketed expressions in the generators,
\[\Mag(\{\ab\}) = \{\ab,\ab\rhd\ab, (\ab\rhd\ab)\rhd \ab, \ab\rhd(\ab\rhd \ab),\dots \} . \]
We identify $\Mag(\mathcal{C})$ with planar binary trees, where the \emph{leaves} are coloured from $\mathcal{C}$ and the internal nodes represent $\rhd$.
$\Alg(\mathcal{C})$ is the linear space where the trees in $\Mag(\mathcal{C})$ form a~basis, called \emph{the monomial basis}.

Another basis for $\Alg(\mathcal{C})$ we call the \emph{tree basis}, is given by $\OT(\mathcal{C})$, the planar (or ordered) trees where \emph{all} nodes are coloured from $\mathcal{C}$. We define $\OT(\mathcal{C})\subset\Alg(\mathcal{C})$ as
\[\OT(\mathcal{C}) = \bigl\{({t_1\dpr t_2\cdots t_k}) \rhd c \mid t_1,\dots,t_k\in \OT(\mathcal{C}),\, k\geq 0,\, c\in \mathcal{C}\bigr\} .\]
The trees are graphically represented as branches $t_1,\dots,t_k$ attached to the root $c$. For $k=0$ we get $\mathcal{C} \subset \OT(\mathcal{C})$ and build up general trees by recursion:
\[\OT(\{\ab\}) = \bigl\{\ab,\aabb,\aaabbb,\aababb,\dots,\overrightarrow{\aabaabbb}, \overrightarrow{\aaabbabb},\dots\bigr\} ,\]
where the arrow indicates that the branches are ordered. The arrow is omitted in the graphical presentation if the tree is symmetric with respect to branch permutations.
We can rewrite from tree basis to monomial basis using~\refp{eq:DA1} and~\refp{eq:DA2}, e.g.,
\[\overrightarrow{\aabaabbb} = \text{ass}\bigl(\ab,\aabb,\ab\bigr) = \text{ass}(\ab,\ab\rhd\ab,\ab) = \ab\rhd((\ab\rhd\ab)\rhd\ab)-(\ab\rhd(\ab\rhd\ab))\rhd\ab .\]
In the tree basis, the product in $\Alg(\mathcal{C})$ is expressed as \emph{left grafting}~\cite{Al-Kaabi_14, MK_Wright_08}. For $t,t_1,\dots,t_k\in \OT(\mathcal{C})$, $c\in \mathcal{C}$, \refp{eq:DA1} and~\refp{eq:DA2} give
\begin{gather*}
 t\rhd(({t_1\dpr t_2\cdots t_k})\rhd c) = {(t\dpr t_1\dpr t_2\cdots t_k})\rhd c + {((t\rhd t_1)\dpr t_2\cdots t_k})\rhd c
+\cdots \\ \hphantom{ t\rhd(({t_1\dpr t_2\cdots t_k})\rhd c) =}{}
+({t_1\dpr t_2\cdots (t\rhd t_k)})\rhd c,
\end{gather*}
as shown in this example
\[\AB\rhd \aababb = \ord{\aABababb}+ \ord{\aaABbabb} + \ord{\aabaABbb} .\]

Below, in Corollary \ref{prop:permute}, we show that any permutation of branches in a tree can be expressed via an iterated triple bracket. The simplest example is
\[\ord{\pABabq}- \ord{\pabABq} = \text{ass}(\AB,\ab,\pq) -\text{ass}(\ab,\AB,\pq) = [\AB,\ab,\pq] .\]

\subsection{Free pre-Lie}Recall that an algebra is pre-Lie if $[\noarg,\noarg,\noarg]\equiv 0$. Thus, the free pre-Lie is the free algebra divided by this relation:
\begin{Definition}[free pre-Lie] The free pre-Lie algebra is $\preLie(\mathcal{C}):= \Alg(\mathcal{C})/\J$, where $\J$ is the ideal generated by the triple bracket.
\end{Definition}
We revisit the classical characterisation of the free pre-Lie algebra~\cite{Chapoton_and_Livernet_01,Dzhumadildeav_and_Lofwall_02}.
Let $\Tree(\mathcal{C})$ be the non-planar trees over $\mathcal{C}$, obtained by forgetting the ordering of branches in $\OT(\mathcal{C})$ and
$(\T_{\mathcal{C}},\vartr)$ the linear span of $\Tree(\mathcal{C})$ with product $\vartr$ being grafting without branch ordering, as
\[\AB\vartr \aababb = {\aABababb}+ 2{\aaABbabb}.\]
In this non-planar case, the enveloping algebra is the symmetric algebra $S(\T_{\mathcal{C}})$.
$S(\T_{\mathcal{C}})$ acts on~$\T_{\mathcal{C}}$ as~\refp{eq:DA1} and~\refp{eq:DA2}, forgetting the ordering.
Define the projection $\pi\colon D(\Alg(\mathcal{C}))\rightarrow S(\T_{\mathcal{C}})$ by forgetting the ordering. This is a homomorphism for both products $\rhd$ and $\cdot$, and is uniquely defined from the universal property of $D(\Alg(\mathcal{C}))$. The following result is a corollary to Proposition~\ref{prop:perm}.

\begin{Corollary}\label{prop:permute} For any $\omega \in D(\Alg(C))$, $|\omega|=k$, $r\in \Alg(\mathcal{C})$ and any permutation $\rho\in S_k$, we~have
\[
\omega\rhd r - \rho(\omega)\rhd r \in \J.
\]
\end{Corollary}

\begin{proof}
 By Proposition \ref{prop:perm}, we have $\omega - \rho(\omega) \in \Tilde{\mathcal{J}} \subset D(\Alg(\mathcal{C}))$. Then $\omega -\rho(\omega)$ must be a~linear combination of elements on the form $u\cdot [y,z]\cdot v$ for $u,v \in D(\Alg(\mathcal{C}))$ and $y,z \in \Alg(\mathcal{C})$. By~(\ref{eq:DA1}) and~(\ref{eq:DA2}), we may rewrite $(u \cdot [y,z] \cdot v) \rhd r$ into something only depending on $\rhd$, and by the observation
 \begin{equation*}
 [y,z] \rhd w = [y,z,w], \qquad \forall \ y,z,w \in \Alg(\mathcal{C}),
 \end{equation*}
 it follows that $\omega \rhd r - \rho(\omega) \rhd r \in \J$.
\end{proof}

\begin{Theorem}[free pre-Lie algebra]\label{th:pre-Lie}
\begin{equation*}\begin{tikzcd}
0\arrow[r] & \J \arrow[r, "\iota"] & \Alg(\mathcal{C}) \arrow[r, "\pi"]& \T_C \arrow[r] &0
\end{tikzcd}
\end{equation*}
is a short exact sequence, hence $\preLie(\mathcal{C}) = (\T_{\mathcal{C}},\vartr)$.
\end{Theorem}

\begin{proof} Since $\pi$ is a homomorphism for $\rhd$ and $\pi([x,y,z])=\pi\bigl((x\dpr y)\rhd z - (y\dpr x)\rhd z\bigr) = 0$, we must have
$\pi(\J)=0$. If $\pi(a) = 0$, then $a$ must be a sum of differences of equivalent trees with permuted branches, and Corollary~\ref{prop:permute} implies $a\in \J$, so the sequence is exact.
\end{proof}

\subsection{Free Lie admissible triple algebras}
The free Lie admissible triple algebra is the quotient $\LAT(\mathcal{C}) = \Alg(\mathcal{C})/\I$, where $\I$ is the two sided ideal generated by~\refp{eq:LAT1} and~\refp{eq:LAT2}. Let $H(S)$ denote a Hall basis for the Lie triple system generated by $S$, henceforth called the ``Hall triple set'', see Definition~\ref{def:H3}.

\begin{Theorem}\label{th:LATbasis}
A basis for $\LAT(\mathcal{C}) = \Alg(\mathcal{C})/\I$ is the smallest subset $\mathcal{B} \subset \Alg(\mathcal{C})$ such that
\[\mathcal{B} = H(S),\]
where
 \[S= \{\s(b_1\cdot b_2\cdots b_k)\rhd c \mid b_1,\dots,b_k\in \mathcal{B},\, c\in \mathcal{C},\, k\geq 0\}\]
are symmetrised trees of basis elements. \end{Theorem}

The proof will be presented in Section \ref{sec:proof}. We present the basis explicitly. For $\mathcal{C}=\{\ab\}$, the basis elements are up to order 5:
\begin{gather*} \mathcal{B} = \bigl\{\ab,{\color{magenta}\aabb},{\color{green}\aababb,\aaabbb},{\color{red}\aabababb,\aabaabbb,\aaababbb,\aaaabbbb, [\aabb,\ab,\ab]},
\\ \hphantom{\mathcal{B} = \{}{}
{\color{blue}\aaabababbb, \aaabaabbbb, \aaaababbbb, \aaaaabbbbb, [\aabb,\ab,\ab]\tr\ab, \aabaababbb, \aabaaabbbb,\aaabbaabbb, \aababaabbb, \aababababb,
[\aaabbb,\ab,\ab],[\aababb,\ab,\ab],[\aabb,\ab,\aabb]},
\dots \bigr\} ,
\end{gather*}
where non-planar trees are identified with the corresponding symmetrised sum of planar trees
\[\aabaabbb := \frac12 \bigl( \ord{\aaabbabb}+\ord{\aabaabbb}\bigr), \qquad \aababaabbb = \frac13 \bigl(\ord{\aaabbababb}+\ord{\aabaabbabb}+\ord{\aababaabbb}\bigr), \qquad\mbox{etc.}\]
The dimensions of the first graded components are given as
$1, 1, 2, 5, 13, 34, 96, 263, \dots$.

From order 5 and higher, we get trees where the \emph{leaves} can be from the Hall triple set:
\[\bigl[\aabb,\ab,\ab\bigr]\rhd \ab,\qquad \mathfrak{s}\bigl(\ab \cdot \bigl[\aabb,\ab,\ab\bigr]\bigr)\rhd \ab,\qquad \mbox{etc.}\]
However, due to~\refp{eq:LAT2}, we never get trees with internal nodes from the Hall triple set.

{\it Summary}: \emph{The basis $\mathcal B$ for $\LAT(\mathcal{C})$ consists of
\begin{itemize}\itemsep=0pt
\item Trees with internal nodes from $C$ and leaves from the Hall triple set of trees.
\item The Hall triple set of trees.
\end{itemize}}

\section{Hall basis for free Lie triple system}\label{sec:Hall}

Recall that an LTS is a vector space with a tri-linear bracket satisfying~\refp{eq:LTS1}--\refp{eq:LTS3}.
We construct a Hall basis for $\LTS(\mathcal{C})$, the free LTS over a set $\mathcal{C}$.
A classical result~\cite{Jacobson1951,Loos1} states that an LTS can always be embedded in a Lie algebra $(\mathfrak{g},[\noarg,\noarg])$ as a subspace
closed under the triple bracket $[a,b,c] := [[a,b],c]$.
In particular, $\LTS(\mathcal{C})$ can be embedded in $\Lie(\mathcal{C})$, the free Lie algebra~\cite{Reutenauer1993free} over $\mathcal{C}$.
$\Lie(\mathcal{C})$ has a $\ZZ_2$ grading where Lie monomials with an odd (resp.\ even) degree have grade~$1$ (resp.\ $0$). The odd monomials span $\LTS(\mathcal{C})$. In particular, any odd monomial can be rewritten as a sum of iterated triple brackets.
For example, using the Jacobi rule and skew symmetry of $[\noarg,\noarg]$, we find
\begin{equation*}
[[[a,b],[c,d]],e] = [a,b,[c,d,e]] - [c,d,[a,b,e]] .
\end{equation*}
Hence any Hall basis for $\Lie(\mathcal{C})$ yields a basis for $\LTS(\mathcal{C})$, by restricting to the odd monomials. However, such bases will in general contain odd monomials which are not simple iterated triple brackets, but must be rewritten as sums of triple brackets.
For example, the Lyndon basis~\cite{Reutenauer1993free} for $\Lie(\{a,b,c,d,e\})$ contains the odd monomial $[[[a,e],[c,d]],b]$.

To obtain a basis for $\LTS(\mathcal{C})$ consisting of only triple brackets, we propose a Hall-type basis obtained directly on $\LTS(\mathcal{C})$, independent of an embedding in $\Lie(\mathcal{C})$. Let $(\Mag_3(\mathcal{C}),[\noarg,\noarg,\noarg])$ denote the \emph{triple magma} over $\mathcal{C}$, i.e., all possible iterated triple brackets formed from $\mathcal{C}$. This can be represented as planar ternary trees with leaves coloured by $\mathcal{C}$. Let $|\noarg|\colon\Mag_3(\mathcal{C})\rightarrow \ZZ$ be the grading $|c|=1$ for $c\in \mathcal{C}$ and $|[a,b,c]| = |a|+|b|+|c|$ for general $a$, $b$, $c$. Let $\Alg_3(\mathcal{C})$ denote the vector space spanned by $\Mag_3(\mathcal{C})$ over some field $K$, and let $\mathcal{M} \subset \Alg_3(\mathcal{C})$ be the ideal generated by the Lie triple relations \refp{eq:LTS1}--\refp{eq:LTS3}.

\begin{Definition}[Hall triple set]\label{def:H3}
We define the Hall triple set $H$ as
\begin{itemize}\itemsep=0pt
\item $\mathcal{C}\subset H\subset \Mag_3(\mathcal{C})$.
\item $H$ has a total ordering $<$ such that $|u|<|v|$ implies $u<v$. The ordering on the same grade is arbitrary, e.g., a lexicographical rule.
\item $[u,v,w]\in H$ if and only if $u,v,w\in H$, $u>v\leq w$ and one of the following hold:
\begin{itemize}\itemsep=0pt
 \item either $u\in \mathcal{C}$,
 \item or $u = [a,b,c]$ where $c\leq v$.
\end{itemize}
\end{itemize}
\end{Definition}

\begin{Proposition}\label{prop:LTS_Hall}
$H$ is a basis for the free Lie triple system $\LTS(\mathcal{C})$.
\end{Proposition}

\begin{proof}
First, we show by induction in the grading that any monomial in $\Alg_3(\mathcal{C})$ can be expressed as a sum of the Hall elements added with terms contained in $\mathcal{M}$.
We introduce a~rewriting based on~\refp{eq:LTS1}--\refp{eq:LTS3}.
Given a monomial $[u,v,w]$. As induction hypothesis, we assume~$u,v,w\in H$.
\begin{enumerate}\itemsep=0pt
\item[(1)] Using~\refp{eq:LTS1} and~\refp{eq:LTS2}, rewrite $[u,v,w]$ so that the smallest element is in the middle:
 \begin{enumerate}\itemsep=0pt
 \item If $u=v$, $[u,v,w]\mapsto 0$.
 \item If $u<v$, $[u,v,w]\mapsto -[v,u,w]$.
 \item If $u>v> w$, $[u,v,w] \mapsto [u,w,v]-[v,w,u]$.
 \end{enumerate}
 After this, we have obtained terms of the form $[u,v,w]$, where $u>v\leq w$. If $u\in \mathcal{C}$ or $u=[a,b,c]$, where $c\leq v$ this is Hall, and we stop.
\item[(2)] For $[u,v,w]$, where $u>v\leq w$, $u=[a,b,c]$ and $c>v$, rewrite using~\refp{eq:LTS3}:
\begin{equation}
[[a,b,c],v,w] \mapsto [a,b,[c,v,w]] - [c,[a,b,v],w] - [c,v,[a,b,w]]. \label{eq:rewrite2}
\end{equation}
Repeat (1) and (2) on the three terms in~\refp{eq:rewrite2} until all terms are Hall.
\end{enumerate}
We need to show that this rewriting will always stop.

Consider the first term on the right in~\refp{eq:rewrite2}, $[a,b,[c,v,w]]$.
If $c\in \mathcal{C}$ or if $c=[c',c'',c''']$ with $c'''\leq v$ then $[a,b,[c,v,w]]\in H$. To see this, recall we have assumed $u=[a,b,c]\in H$, so $a = [a',a'',a''']$ with $a'''>b$ is not possible. Furthermore, $b<a$ and $b\leq c$ implies $b< [c,v,w]$, and we have $c>v\leq w$, which shows that it is Hall. If $c=[c',c'',c''']$ with $c> v$, we use the induction hypothesis to write $[c,v,w]$ as a sum of
Hall elements $h_i\in H$, $h_i>b$. This yields Hall elements $[a,b,h_i]\in H$.

Now we continue to rewrite the two rightmost terms in~\refp{eq:rewrite2}. We will argue that the Hall elements
$[a,b,h_i]\in H$ produced from the first term cannot reappear in the rewriting of the last two terms. If this is so, each round of rewriting removes some Hall words completely from the remaining terms, and the iteration must eventually stop.

To see that $[a,b,h_i]\in H$ will not reappear, we first consider the second term $[c,[a,b,v],w]$. Rewriting could give
$[[a,b,v],c,w]$ or $[[a,b,v],w,c]$. However, since $v<c$ and $v\leq w$, we will never use~\refp{eq:rewrite2} and get back terms of the form $[a,b,[\dots]]$.

The last term is $[c,v,[a,b,w]]$, where $v<c$ and $v\leq w$, thus $v<[a,b,w]$. The smallest term is $v$ in the middle, so we
will not use rewriting (c)
and get back terms with $a$, $b$ moved to the left, hence $[a,b,[\dots]]$ can not reappear.

This shows that $H$ spans $\LTS(\mathcal{C})$. To see that $H$ is a basis, we note that the rewriting rules guarantee that~\refp{eq:LTS1}--\refp{eq:LTS3} hold for all $a,b,x,y,z\in H$. By tri-linearity, these relations must also hold in the span of $H$, and hence $H$ is a basis for the quotient $\LTS(\mathcal{C}) = \Alg_3(\mathcal{C})/\langle \refp{eq:LTS1},\refp{eq:LTS2},\refp{eq:LTS3}\rangle$.
\end{proof}

\begin{Definition}\label{def:Hall_map}
Given a totally ordered set $\mathcal{C}$, let $\mathfrak{H} \colon \Alg_3 (\mathcal{C}) \rightarrow \Alg_3(\mathcal{C})$ be the linear projection that rewrites any element of $\Mag_3(\mathcal{C})$ into a sum of Hall elements according to the algorithm described in the proof above, i.e., such that $\ker(\mathfrak{H}) = \mathcal{M}$ and $Im(\mathfrak{H}) = \LTS(\mathcal{C})$.
\end{Definition}

For $\LTS(\{a,b\})$, the Hall triple set up to order 7 is
\begin{align*}
H = \{{}&a, b,
[baa], [bab],
[[baa]aa], [[baa]ab],
[[baa]bb], [[bab]bb], [ba[baa]],
[ba[bab]],
[[[baa]aa]aa], \\ &{} [[[baa]aa]ab], [[[baa]aa]bb],
[[[baa]ab]bb], [[[baa]bb]bb], [[[bab]bb]bb],
[ba[[baa]aa]], \\ &{}
[ba[[baa]ab]],[ba[[baa]bb]], [ba[[bab]bb]],
[ba[ba[baa]]], [ba[ba[bab]]],
[[baa]a[baa]],
\\ &{}
[[baa]b[baa]],
[[bab]b[baa]], [[baa]a[bab]],
[[baa]b[bab]], [[bab]b[bab]], \dots\} .
\end{align*}
The dimensions of the graded components, $2, 2, 6, 18, 56, 186, \dots$ are computed by the classical Witt formula for odd grades, see \href{https://oeis.org/A001037}{A001037}.

\section{OSBB-words: a basis for the free tensor algebra}\label{sec:OSBB}

In this section, consider a vector space $W$ over a field $K$ with a totally ordered basis $\mathcal{C}$, and let~$\prec$ denote the total order on $\mathcal{C}$. Let $T(W)$ be the free tensor algebra over $W$, equipped with the structure and notation introduced in Section \ref{sec:tensor}.

Recall from Section \ref{sec:Trees} that one way to generate the planar rooted trees is to \emph{graft} an \emph{ordered forest} (or a word) onto a \emph{root}. If instead this process is done by only grafting \emph{unordered forests} $\mathfrak{s}(x_1 \cdots x_k)$ onto a root, we would generate the non-planar rooted trees, which is a basis for the free pre-Lie algebra. This algebra is defined by the relation $[\noarg, \noarg, \noarg]=0$, hence the triple-brackets together with non-planar trees should in some sense be sufficient to describe the full free algebra $\Alg(\mathcal{C})$. The triple-bracket can be described by grafting a bracket $[x,y]$ onto a tree; $[x,y]\rhd z = [x,y,z]$.

The motivation to continue this journey of combinatorics on the tensor algebra is to be able to generate a basis for the free algebra using triple-brackets and non-planar trees instead of planar trees, which corresponds to a basis on the free tensor algebra using symmetrized elements and brackets. This alternative basis on the tensor algebra is presented in Theorem \ref{th:OSBB} below, after we have developed some language.
\begin{itemize}\itemsep=0pt
 \item \textit{Symmetric word:} Let $x_1, \dots, x_k$, $k\geq 0$ be elements from the set $\mathcal{C}$. The element $\mathfrak{s}(x_1 \cdots x_k)$ is called a \emph{symmetric word}. Note that the empty word is a symmetric word.
 \item \textit{Symmetric-bracket word:} If $y,z \in \mathcal{C}$, then the element $\mathfrak{s}(x_1 \cdots x_k) \cdot [y,z]$ is called a~\emph{sym\-met\-ric-bracket word}.
 \item \textit{SBB-word} An element $\omega = \omega_1 \cdots \omega_s$ is called a \emph{symmetric-bracket-block word}, or an SBB-word, if $\omega_i$ is a symmetric-bracket word for $1\leq i <s$ and $\omega_s$ is a symmetric word. As long as $\omega_s$ is non-empty, we say that $\omega$ has $s$ \emph{blocks}, where $\omega_i$ is a block of $\omega$ for $1\leq i \leq s$. If~$\omega_s$ is the empty word, $\omega$ has $s-1$ blocks.
 \item \textit{OSB-word:} A symmetric-bracket word $\mathfrak{s}(x_1 \cdots x_k) \cdot [y,z]$ is \emph{ordered} if $z\prec y$ and $z \preceq x_i$ for all $i$, where $\prec$ is the given total order on $\mathcal{C}$. We abbreviate \emph{ordered symmetric-bracket word} to OSB-word.
 \item \textit{OSBB-word:} An element $\omega = \omega_1 \cdots \omega_s$ is called an \emph{ordered symmetric-bracket-block word}, or an OSBB-word, if $\omega$ is an SBB-word, and in addition $\omega_i$ is ordered for $1\leq i <s$.
\end{itemize}

We make a clarifying note on the use of the word \emph{word}. A word is an element of the natural basis of $T(W)$. The list above contains five special types of elements of homogeneous degree in~$T(W)$ which are strictly speaking not words, but since they are closely associated concepts, we have chosen to name them this way. 

\begin{Theorem}\label{th:OSBB}
 Let $\mathcal{C}$ be a basis of some vector space $W$, and let $\prec$ be a total order on $\mathcal{C}$. Then the OSBB-words over $\mathcal{C}$ form a basis of the free tensor algebra $T(W)$.
\end{Theorem}

Before we are ready to prove this theorem, we need to prove the weaker claim that any word in $T(W)$ can be written as a linear combination of SBB-words.

\begin{Proposition}\label{cor:SBB-word}
Any word can be rewritten as a sum of SBB-words.
\end{Proposition}

\begin{proof}
 This is immediately true for words of length 1. Assume true for words of length less than $k$ and let $\omega \in T(W)$ be any word of length $k$. Then $\omega - \s(\omega) \in \Tilde{\mathcal{J}}$ by Proposition \ref{prop:perm}. We~get
 \[
 \omega = \s(\omega) + (\omega - \s(\omega)) = \s(\omega) + \nu,
 \]
 where $\nu$ is a linear combination of elements on the form $\nu_1 \cdot [x,y] \cdot \nu_2$ with both $|\nu_1|<k$ and $|\nu_2|<k$, hence we can rewrite them as a sum of SBB-words by assumption.
\end{proof}

\begin{proof}[Proof of Theorem \ref{th:OSBB}]
 \textit{Step $1$.} Prove that any word can be written as a sum of OSBB-words. By Proposition \ref{cor:SBB-word}, it is sufficient to show that this is possible for any SBB-word. Assume this is true for SBB-words of length less than $k$ and let $\omega$ be a SBB-word of length $k$. Then $\omega= \omega_1 \cdots \omega_s$. Assume that $\omega_1 = \s(x_1 \cdots x_n \cdot u) \cdot [y,z]$ with $u \prec z \prec y$. The Jacobi identity gives
 \[
 u \cdot[y,z] = [y,z] \cdot u - y \cdot [z, u] + [z,u]\cdot y + z \cdot [y,u] - [y,u]\cdot z =: J(u \prec z \prec y).
 \]
 Then we can write
 \[
 \begin{split}
 \omega_1 &= \frac{1}{n+1} \Biggl( \s(x_1 \cdots x_n) \cdot u \cdot [y,z] + \sum_{i=1}^n \s(x_1 \cdots \hat{x}_i \cdots x_n \cdot u)\cdot x_i \cdot [y,z] \Biggr) \\
 &= \frac{1}{n+1} \s(x_1 \cdots x_n) \cdot J(u \prec z \prec y) + \sum_{i=1}^n \biggl( \frac{1}{n+1} \omega_1 + \nu_i(x_1 \cdots x_n \cdot u) \cdot [y,z]\biggr),
 \end{split}
 \]
 where $\nu_i(x_1 \cdots x_n \cdot u) \in \Tilde{\mathcal{J}}$ is the terms you get from rewriting $\s(x_1 \cdots \hat{x}_i \cdots x_n \cdot u) \cdot x_i - \s{(x_1 \cdots x_n \cdot u)}$ according to Proposition \ref{prop:perm}. Solving for $\omega_1$, we get
 \begin{align}
 \omega_1 &= \s(x_1 \cdots x_n) \cdot J(u \prec z \prec y) + \sum_i \nu_i(x_1 \cdots x_n \cdot u) \cdot [y,z]\label{eq:OSBB_rewrite}\\
 &= \s(x_1 \cdots x_n) \cdot ( [y,z] \cdot u - y \cdot [z, u] + [z,u]\cdot y + z \cdot [y,u] - [y,u]\cdot z ) + \sum_i \nu_i \cdot [y,z].\nonumber
 \end{align}
 We need to show that each of these parts resolve into OSBB-words when concatenated with $\omega_2 \cdots \omega_s$.
 \begin{itemize}\itemsep=0pt
 \item The terms $[z,u]\cdot y$ and $[y,u]\cdot z$ in $J(u\prec z\prec y)$ are similar and give words on the form $\eta \cdot p \cdot\omega_2 \cdots \omega_s$, where $p$ is either $y$ or $z$. Notice that $\eta$ is an OSB-word and $|p \cdot \omega_2 \cdots \omega_s|<k$ and can therefore be written as a OSBB-word by assumption.
 \item The term $[y,z]\cdot u$ from $J(u \prec z \prec y)$ gives us a SB-word $\s(x_1 \cdots x_n)\cdot [y,z]$ connected to a word $u \cdot \omega_2 \cdots \omega_s$. The latter part has length smaller than $k$ and can be written as a OSBB-word by assumption. If $\s(x_1 \cdots x_n)\cdot [y,z]$ is not ordered, we find the smallest element $x_i$ and rewrite as above, but for $x_i$ instead of $u$. Since repeating this process reduces the length of the word by one, hence the first block will resolve into an OSB-word in the end, while the rest is a word of length shorter than $k$ and can therefore be rewritten as an OSBB-word.
 \item Consider $y \cdot [z,u]$. By Proposition~\ref{prop:perm}, we get $\s(x_1 \cdots x_n) \cdot y \cdot [z,u] = \s(x_1 \cdots x_n \cdot y) \cdot [z,u] + \eta \cdot [z,u]$. The first term is OSB. We can write $\eta$ as a sum of SBB-words by Proposition~\ref{cor:SBB-word}, $\eta = \eta_1 \cdots \eta_r$. Since $|\eta_2 \cdots \eta_r \cdot [z,u] \cdot \omega_2 \cdots \omega_s|<k$, we can rewrite this as a sum of OSBB-words by assumption. If $\eta_1$ is not ordered, we detect the smallest element and redo this process. Since this reduces the length of the first block, the process must terminate.
 The case $z \cdot [y,u]$ is the same.
 \item Lastly, consider $\nu_i(x_1 \cdots x_n \cdot u)$. We rewrite this as a sum of SBB-words $\eta_1 \cdots \eta_r$. Again $\eta_2 \cdots \eta_r \cdot [y,z] \cdot \omega_2 \cdots \omega_s$ has length shorter than $k$ and can be rewritten as an OSBB-word by assumption. If $\eta_1$ is not OSB, we detect the smallest element and redo this process. Since the length of $\eta_1$ is reduced by each iteration, it must terminate.
 \end{itemize}

 \textit{Step $2$.} Prove that the rewriting is unique. Let $\alpha_k(n)$ denote the number of OSBB-words of length $k$ made from $n$ letters where the letters can be used any number of times. For convenience, let $\alpha_0(n)=1$ for all $n$. We need to show that $\alpha_k(n) = n^k$. First, note the following recursive formula for $\alpha_k(n)$:
 \begin{equation}\label{eq:rec_alpha}
 \alpha_k(n) = \binom{k + (n-1)}{(n-1)} + \sum_{i=2}^k \Biggl( \sum_{j=1}^{n-1} (n-j) \binom{(i-2) + (n-j)}{(n-j)} \Biggr) \alpha_{k-i}(n).
 \end{equation}
 Each part is explained as follows:
 \begin{itemize}\itemsep=0pt
 \item $\binom{k + (n-1)}{(n-1)}$ is the number of symmetric words we can make of length $k$ from $n$ letters.
 \item The $i$ denotes the length of the first block in an OSBB-word that is not symmetric, hence the shortest possible length of the first block is $2$ in the case it is just a bracket.
 \item The $j$ denotes the number of the last letter in the bracket of the first block with the enumeration being given by the order on the letters, counting from the smallest letter as number $1$, and the largest letter is then number $n$. Notice that $j$ only counts until $n-1$, since there must at least be one larger letter available to put in the bracket.
 \item $(n-j) \binom{(i-2) + (n-j)}{(n-j)}$ is the number of OSB-words of length $i$ with the last letter of the bracket being number $j$ counting from smallest to largest. The factor $(n-j)$ corresponds to choosing a larger letter to be the first letter in the bracket, and $\binom{(i-2) + (n-j)}{(n-j)}$ corresponds to choosing letters for the symmetric part, all of which is larger or equal to the last letter of the bracket.
 \item If we ignore the $n$, we can read the formula as: The number of OSBB-words of length~$k$ is equal to the number of symmetric words of length $k$ plus the number OSB-words of length~$i$ times the number of OSSB-words of length $k-i$ for each $2\leq i \leq k$, that is, for each possible length of the first block.
 \end{itemize}
 We want to prove that $\alpha_k(n)=n^k$ by induction on $k$. First, note that clearly $\alpha_1(n) = n$. Assume that $\alpha_r(n)=n^r$ for all $r<k$, and note that by \eqref{eq:rec_alpha}, we have
 \begin{equation}\label{eq:rec_alpha2}
 \alpha_{k-1}(n) = \binom{(k-1) + (n-1)}{(n-1)} + \sum_{i=2}^{k-1} \Biggl( \sum_{j=1}^{n-1} (n-j) \binom{(i-2) + (n-j)}{(n-j)} \Biggr) \alpha_{k-1-i}(n).
 \end{equation}
 By assumption, $\alpha_{k-i}(n) = n^{k-i}$ as long as $i$ is positive, hence $n\alpha_{k-1-i}(n) = \alpha_{k-i}(n)$. We note that the coefficient of $\alpha_{k-1-i}(n)$ in \eqref{eq:rec_alpha2} is equal to the coefficient of $\alpha_{k-i}(n)$ in \eqref{eq:rec_alpha}. These observations let us rewrite $\alpha_k(n)$ as
 \begin{gather*}
 \alpha_k(n) = n \alpha_{k-1}(n) + \sum_{j=1}^{n-1}(n-j)\binom{(k-2) + (n-j)}{(n-j)} + \binom{k + (n-1)}{(n-1)}\\ \hphantom{\alpha_k(n) =}{}
 - n\binom{(k-1)+(n-1)}{(n-1)}.
 \end{gather*}
 Here $\sum_{j=1}^{n-1}(n-j)\binom{(k-2) + (n-j)}{(n-j)}$ is the term in \eqref{eq:rec_alpha} corresponding to the term in the summation when $i=k$ which can not be captured in $\alpha_{k-1}(n)$, and the last two terms correspond to the first term in \eqref{eq:rec_alpha} and the first term in \eqref{eq:rec_alpha2}. By assumption, $n\alpha_{k-1}(n)=n^k$, and the rest of the terms cancel by Lemma \ref{lemma:bin_ind}, hence we have $\alpha_k(n)=n^k$, and the theorem follows by induction.
\end{proof}

\subsection{An order on the OSBB-words}

To construct OSBB-words, we depend on a total order on the elements we construct the words with. We want to use OSBB-words to construct a basis for the free algebra, but then we must construct a total order on the basis simultaneously. We start by constructing a general order on the OSBB-words, depending on the order of the elements the words are made from.

\begin{Definition}
 Let $W$ be a vector space with a basis $\mathcal{C}$ and a total order $\prec = \prec_{\mathcal{C}}$ on $\mathcal{C}$. Define a total order $\prec_{\Delta} = \prec_{\Delta_{\mathcal{C}}}$ on the set of OSBB-words in $T(W)$ by
 \begin{enumerate}\itemsep=0pt
 \item[(i)] \textit{Compare symmetric words:} Let $\omega = \mathfrak{s}(x_1 \cdots x_n)$ and $\eta = \mathfrak{s}(y_1 \cdots y_m)$ be two symmetric words and let the indexing be such that $x_1 \succeq \dots \succeq x_n$ and $y_1 \succeq \dots \succeq y_m$. We set $\omega \succ_{\Delta} \eta$~if
 \begin{itemize}\itemsep=0pt
 \item $n>m$, or
 \item $|\omega| = |\eta|$, $x_i = y_i$ for $i<j$ and $x_j \succ y_j$.
 \end{itemize}
 \item[(ii)] \textit{Compare OSB-words:} Let $\omega = \mathfrak{s}(x_1 \cdots x_n)\cdot[u,v]$ and $\eta = \mathfrak{s}(y_1 \cdots y_m)\cdot[w,z]$ be two ordered symmetric-bracket words and let the indexing be such that $x_1 \succeq \dots \succeq x_n$ and $y_1 \succeq \dots \succeq y_m$. We set $\omega \succ_{\Delta} \eta$ if
 \begin{itemize}\itemsep=0pt
 \item $n>m$, or
 \item $n=m$ and $v \succ z$, or
 \item $n=m$, $v=z$ and $u\succ w$, or
 \item $n=m$, $v=z$, $u=w$, and $\mathfrak{s}(x_1 \cdots x_n) \succeq_{\Delta} \mathfrak{s}(y_1 \cdots y_m)$.
 \end{itemize}
 \item[(iii)] \textit{Compare symmetric word and OSB-word:} Let $\omega = \mathfrak{s}(x_1 \cdots x_n)\cdot[u,v]$ be an OSB-word and $\eta = \mathfrak{s}(y_1 \cdots y_m)$ a symmetric word, again assuming the $x_i$'s and $y_i$'s are listed in decreasing order. Then $\omega \succ_{\Delta} \eta$ if and only if
 \begin{itemize}\itemsep=0pt
 \item $|\omega| > |\eta|$.
 \end{itemize}
 \item[(iv)] \textit{Compare OSBB-words:} let $\omega = \omega_1 \cdots \omega_k$ and $\eta = \eta_1 \cdots \eta_r$ be two OSBB-words. Let $\omega \succ_{\Delta} \eta$ if
 \begin{itemize}\itemsep=0pt
 \item $|\omega| > |\eta|$, or
 \item $|\omega|=|\eta|$, $\omega_i = \eta_i$ for $i<j$ and $\omega_j \succ_{\Delta} \eta_j$.
 \end{itemize}
 \end{enumerate}
\end{Definition}

The following lemma is a generalization of a normal quality of the lexicographical order on words: if one letter in a word is replaced by a smaller letter (lets say we replace an ``A'' with a ``B'') the resulting word would be smaller. When considering OSBB-words, the case becomes a bit more technical, since replacing letters with smaller letters means the resulting word need not be an OSBB-word.

An OSBB-word $\omega$ will consist of $|\omega|$ elements from the underlying algebra put together by $\mathfrak{s}$, $[\noarg , \noarg]$ and $\cdot$. If $\omega = \mathfrak{s}\bigl(x_1^1 \cdots x_{n_1}^1\bigr) \cdot\bigl[y^1, z^1\bigr] \cdots \mathfrak{s}\bigl(x^k_1 \cdots x^k_{n_k}\bigr)$ we introduce the notation $\Bar{\omega}\bigl(u_1 , \dots , u_{|\omega|}\bigr) = \mathfrak{s}\bigl(u_1 \cdots u_{n_1}\bigr) \cdot \bigl[u_{n_1 +1}, u_{n_1 +2}\bigr] \cdots \mathfrak{s}\bigl(u_{|\omega|-n_k} \cdots u_{|\omega|}\bigr)$, by which we mean: maintain the same structure of the word with respect to $\mathfrak{s}$, $[\noarg , \noarg]$ and $\cdot$, but exchange the $i$'th input with~$u_i$. Notice that $\Bar{\omega}\bigl(u_1 , \dots , u_{|\omega|}\bigr)$ need not be an OSBB-word since the condition that the second argument of each bracket must be minimal in its block might not be satisfied.

\begin{Lemma}\label{lemma:smaller_OSBB}
 Let $\omega = \Bar{\omega}\bigl(u_1 , \dots , u_{|\omega|}\bigr)$ be an OSBB-word, and let $v_i \preceq u_i$ for all $i$. Rewrite $\Bar{\omega}\bigl(v_1 , \dots , v_{|\omega|}\bigr)$ as a linear combination of OSBB-words. Then each term $\eta$ in this linear expansion satisfies $\eta \preceq_{\Delta} \omega$.
\end{Lemma}

\begin{proof}
 We do this by induction on the length $|\omega|$. The statement is clearly true for $|\omega|=1$, and assume it is true whenever $|\omega|<k$. Consider the case $|\omega|=k$. We introduce the notation $\{u_i \leftrightarrow v_i \}$ to mean that $u_i$ is replaced by $v_i$. Notice that rewriting $\Bar{\omega}\bigl(u_1, \dots , \{ u_i \leftrightarrow v_i \} , \dots ,u_{|\omega|}\bigr)$ as a sum of OSBB-words $\eta_j$, then changing $\{u_l \leftrightarrow v_l \}$ in each of the OSBB-words $\eta_j$ and then rewriting all those terms as OSBB-words will give the same result as rewriting $\Bar{\omega}\bigl(u_1, \dots , \{ u_i \leftrightarrow v_i \} , \dots , \{ u_l \leftrightarrow v_l \} , \dots ,u_{|\omega|}\bigr)$ into a linear combination of OSBB-words directly. This is because we have equality at each step, and the rewriting into OSBB-words is unique. Hence it is sufficient to show that the lemma holds when $u_i = v_i$ for all $i$ except one. Furthermore, if $\omega = \omega_1 \cdots \omega_s$ and we exchange an element from another block than $\omega_1$, by assumption the lemma holds for the OSBB-word $\omega_2 \cdots \omega_s$, and we get $\eta = \omega_1 \cdot \eta_0$ where $\eta_0$ is the result after solving the reduced problem. Then $\eta_0 \preceq_{\Delta} \omega_2 \cdots \omega_s$ implies $\eta \preceq_{\Delta} \omega$.

 Let $\omega_1 = \mathfrak{s}(x_1 \cdots x_n)\cdot[y,z]$.
 \begin{itemize}\itemsep=0pt
 \item Let $\Tilde{z}\preceq z$, then $\Tilde{\omega}_1 \cdot \mathfrak{s}(x_1 \cdots x_n)\cdot[y,\Tilde{z}]$ is an ordered-symmetric-bracket block and $\Tilde{\omega}_1 \cdot \omega_2 \cdots \omega_s \preceq_{\Delta} \omega$.
 \item Let $\Tilde{y} \preceq y$, then $\Tilde{\omega}_1 \cdot \mathfrak{s}(x_1 \cdots x_n)\cdot[\Tilde{y},z]$ is either an OSB-word, equal to zero (if $\Tilde{y}=z$), or $\Tilde{\omega}_1 = - \mathfrak{s}(x_1 \cdots x_n)\cdot[z,\Tilde{y}]$ is an OSB-word. In all cases $\Tilde{\omega}_1 \cdot \omega_2 \cdots \omega_s \preceq_{\Delta} \omega$.
 \item If $\Tilde{x}_i \preceq x_i$, we must rewrite $\mathfrak{s}(x_1 \cdots \{ x_i \leftrightarrow \Tilde{x}_i\} \cdots x_n)\cdot[y,z]\cdot \omega_2 \cdots \omega_s$ as a sum of OSBB-words. If $\Tilde{x}_i \succeq z$, we are done. If not, we must consider equation (\ref{eq:OSBB_rewrite}) from the proof of Theorem \ref{th:OSBB}. By inspection, we can see that the first block after rewriting will either have~$\Tilde{x}_i$ as the last term of the bracket or have shorter length than $\omega_1$, and we get $\eta \preceq_{\Delta} \omega$.
 \end{itemize}
 This concludes the proof.
\end{proof}

\section{Using OSBB-words to build bases for the free algebra}\label{sec:bases}

The basis $\mathcal{B}$ of $\LAT(\mathcal{C})$ presented in Theorem \ref{th:LATbasis} is not naturally compatible with the tree basis $\OT(\mathcal{C})$ of the free algebra. In fact, it is not a priori clear that the elements of $\mathcal{B}$ are linearly independent in $\Alg(\mathcal{C})$ as the elements in $\mathcal{B}$ could be very complicated to write in the basis $\OT(\mathcal{C})$. By using Theorem \ref{th:OSBB}, we will be able to construct a basis for $\Alg(\mathcal{C})$ in which $\mathcal{B}$ is a~subset, hence we conclude at the end of this section that the suggested basis for $\LAT(\mathcal{C})$ is at least a linearly independent set.

\subsection[An ordered basis of Alg(C)]{An ordered basis of $\boldsymbol{\Alg(\mathcal{C})}$}

The tree basis for the free algebra $\Alg(\mathcal{C})$ is the set of planar rooted trees colored by the set~$\mathcal{C}$. Any tree $\tau$ in this basis will have a specific number of vertices, and we let $|\tau|_V$ denote this number: The number of vertices $| \noarg |_V$ of a tree in $\OT(\mathcal{C})$ is given by
\begin{itemize}\itemsep=0pt
 \item $|c|_V=1$ for any $c \in \mathcal{C}$, and
 \item $|(t_1 \cdots t_n) \rhd c|_V = 1 + |t_1|_V + \dots + |t_n|_V$ for $t_1 , \dots , t_n \in \OT(\mathcal{C})$.
\end{itemize}

Then the free algebra $\Alg(\mathcal{C})$ is a naturally graded vector space
\[
\Alg(\mathcal{C}) \cong \bigoplus_{i=1}^{\infty} V_i,
\]
where $V_k$ is the vector space spanned by all the trees with $k$ vertices. This definition can immediately be generalized to the standard basis elements of the D-algebra by $|t_1 \cdots t_n|_V = |t_1|_V + \dots + |t_n|_V$. Furthermore, it is clear that rewriting an element $\omega$ from the standard basis preserves both the length of the word $|\omega|$ and the number of vertices $|\omega|_V$, hence $|\noarg|_V$ is well defined for OSBB-words of elements on which $|\noarg|_V$ is well defined. Lastly, if $\omega$ is an OSBB-word and $c \in \mathcal{C}$, we may define $|\omega \rhd c|_V = 1+|\omega|_V$, something that corresponds to the notion that $\omega \rhd c \in V_k$ where $k = 1+ |\omega|_V$.

Let $\mathcal{C}$ be a totally ordered set. Consider the following set $\mathcal{S} \subset \Alg(\mathcal{C})$:
\begin{itemize}\itemsep=0pt
 \item if $c\in \mathcal{C}$ then $c \in \mathcal{S}$,
 \item if $\omega$ is an OSBB-word of elements from $\mathcal{S}$, then $\omega \rhd c \in \mathcal{S}$,
\end{itemize}
where we use the order $\prec_{\mathcal{S}}$ defined by $\omega \rhd c_1 \prec_{\mathcal{S}} \eta \rhd c_2$ if
\begin{enumerate}\itemsep=0pt
 \item[(i)] $|\omega|_V < |\eta|_V$, or
 \item[(ii)] $|\omega|_V = |\eta|_V$ and $\omega \prec_{\Delta} \eta$, or
 \item[(iii)] $\omega = \eta$ and $c_1 < c_2$.
\end{enumerate}

\begin{Proposition}
 The set $\mathcal{S}$ is a basis for $\Alg(\mathcal{\mathcal{C}})$.
\end{Proposition}

\begin{proof}
 By induction on $|\noarg|_V$. Since $c \in \mathcal{S}$ for any $c \in \mathcal{C}$, and this is a basis for $V_1$, we see that the elements in $\mathcal{S}$ with $|x|_V=1$ is indeed a basis for $V_1$. Assume that the elements $y \in \mathcal{S}$ with $|y|_V= j$ is an ordered basis for $V_j$ for all $j<k$. Thus by assumption, any tree in the tree basis of $V_j$ can be written uniquely as a linear combination of elements in $\mathcal{S}$.

 \textit{Step $1$.} Show that any element in $V_k$ can be written uniquely as a linear combination of elements in $\mathcal{S}$. Let $x \in V_k$ be a planar rooted tree in the tree basis for $\Alg(\mathcal{C})$. Then $x = (t_1 \cdots t_n) \rhd c$. By assumption, each $t_i$ can be uniquely rewritten as a linear combination of~$\mathcal{S}$ giving $x = \sum_i a_i (\omega_i \rhd c)$, where $a_i \in \mathbb{R}$ and $\omega_i$ are words of elements form $\mathcal{S}$ satisfying $|\omega_i|=n$ and $|\omega_i|_V = k-1$. By Theorem \ref{th:OSBB}, each of these words $\omega_i$ can be uniquely rewritten as a~linear combination of OSBB-words, since $\omega_i$ is a word of elements $y^i_j$ from $\mathcal{S}$ with $\bigl|y^i_j\bigr|_V <k$, and these are ordered by assumption. this means that any element in $V_k$ can be uniquely written as a linear combination of elements from $\mathcal{S}$.

 \textit{Step $2$.} Show that the order $\prec_{\mathcal{S}}$ is well defined on elements in $V_k$. Let $x , y\in \mathcal{S}$ be two elements with $|x|_V = |y|_V = k$. Then $x = \omega \rhd c_1$ and $y= \eta \rhd c_2$ for some OSBB-words~$\omega$ and~$\eta$ of elements in $\mathcal{S}$ with less than $k$ vertices. We may compare $x$ and $y$ by comparing~$\omega$ and~$\eta$ using $\prec_{\Delta}$ and by comparing $c_1$ and $c_2$ using the given total order on $\mathcal{C}$. Notice that the order on $\mathcal{S}$ prioritize the number of vertices, hence $x$ is clearly comparable to any $z \in V_j$ for $j<k$.

 By induction on the number of vertices, the elements $x$ of $\mathcal{S}$ with $|x|_V=i$ is an ordered basis of~$V_i$ for any $i$, hence $\mathcal{S}$ is a totally ordered basis for $\Alg(\mathcal{C})$.
\end{proof}

\subsection[An ordered basis for Alg(C) that contains B]{An ordered basis for $\boldsymbol{\Alg(\mathcal{C})}$ that contains $\boldsymbol{\mathcal{B}}$}

We will now define another basis for the free algebra that is the one that most closely capture the important features for understanding the Lie admissible triple algebra. Although the basis~$\mathcal{S}$ does give an understanding of the intrinsic relations that follows from the definition of the triple-bracket, it is not a convenient basis when working with the LAT-relations.

\begin{Definition}
 Define a map $\psi$ by
 \begin{align*}
 \psi\colon \ \mathcal{S} &\longrightarrow \Alg(\mathcal{C}), \\
 c &\longmapsto c ,\\
 \mathfrak{s}(x_1 \cdots x_n) \rhd c &\longmapsto \mathfrak{s}(\psi(x_1) \cdots \psi(x_n)) \rhd c, \\
 (\mathfrak{s}(x_1 \cdots x_n) \cdot [y, z] \cdot \omega) \rhd c & \longmapsto \mathfrak{s}(\psi(x_1) \cdots \psi(x_n)) \rhd [\psi(y), \psi(z), \psi(\omega \rhd c)]
 \end{align*}
 and extend to $\Alg(\mathcal{C})$ by linearity.
\end{Definition}

\begin{Proposition}
 The map $\psi$ is a linear automorphism of $\Alg(\mathcal{C})$.
\end{Proposition}

\begin{proof}
 To prove this, we start by showing that $\psi(x) = x + q$ where $q$ is a linear combination of elements in $\mathcal{S}$ that is smaller than $x$. This is clearly true for $c \in V_1$. Assume this is true for $y \in V_j$ with $j<k$, and let $x \in V_k$. There are two cases:
 \begin{itemize}\itemsep=0pt
 \item If $x = \mathfrak{s}(x_1 \cdots x_n) \rhd c$, we have $\psi(x) = \mathfrak{s}(\psi(x_1) \cdots \psi(x_n)) \rhd c$. By assumption, $\psi(x_i) = x_i +q_i$, and since all the terms in $q_i$ are smaller than $x_i$ by assumption, we immediately have $\psi(x) = x + q$, where all the terms in $q$ are smaller than $x$.
 \item If $x = (\mathfrak{s}(x_1 \cdots x_n) \cdot [y, z] \cdot \omega) \rhd c$, we have $\psi(x) = \mathfrak{s}(\psi(x_1) \cdots \psi(x_n)) \rhd [\psi(y), \psi(z), \psi(\omega \rhd c)]$. By assumption, $\psi(x_i) = x_i + q_i$, $\psi(y) = y + p$, $\psi(z) = z + r$ and $\psi(\omega \rhd c) = \omega \rhd c + s$. It is clear that the terms in $s$ has the form $\eta \rhd c$. We need to show that if $\Tilde{x}_i\preceq x_i$, $\Tilde{y} \preceq y$, $\Tilde{z} \preceq z$ and $\eta \preceq_{\Delta} \omega$, then
 \[
 \mathfrak{s}(\Tilde{x}_1 \cdots \Tilde{x}_n) \cdot [\Tilde{y}, \Tilde{z}] \cdot \eta \preceq_{\Delta} \mathfrak{s}(x_1 \cdots x_n) \cdot [y, z] \cdot \omega.
 \]
 This follows from Lemma \ref{lemma:smaller_OSBB}.
 \end{itemize}
 This means that the map $\psi$ acts as an upper triangular matrix with $1$'s on the diagonal on each graded component $V_i$ with respect to the ordered basis $\mathcal{S}$. Clearly such a matrix is invertible, hence $\psi$ is an automorphism.
\end{proof}

Let $\mathcal{D}$ be defined as the image of $\mathcal{S}$ by the map $\psi$. Then $\mathcal{D}$ is a basis of $\Alg(\mathcal{C})$. We may define $\mathcal{D}$ iteratively by
\begin{enumerate}\itemsep=0pt
 \item[(i)] if $c \in \mathcal{C}$, then $c \in \mathcal{D}$,
 \item[(ii)] if $x = \mathfrak{s}(x_1 \dots x_n) \rhd c$, where $x_i \in \mathcal{D}$, then $x \in \mathcal{D}$,
 \item[(iii)] if $x = \mathfrak{s}(x_1 \dots x_n) \rhd [y_1, y_2, y_3]$, where $x_i, y_j \in \mathcal{D}$ with $y_1 > y_2$ and $y_2 \leq x_i$ for $i=1, \dots , n$ and $j=1,2,3$, then $x \in \mathcal{D}$.
\end{enumerate}
It is sufficient to compare this definition with the image of $\psi$ to see that the two definitions of~$\mathcal{D}$ are the same. The set $\mathcal{D}$ is a totally ordered set by $\psi(x) \prec \psi(y)$ if and only if $x \prec_{\mathcal{S}} y$.

\begin{Corollary}\label{cor:B_lin_ind}
 The set $\mathcal{B}$ from Theorem $\ref{th:LATbasis}$ is linearly independent in $\Alg(\mathcal{C})$.
\end{Corollary}

\begin{proof}
 Observe that $\mathcal{B} \subset \mathcal{D}$.
\end{proof}

\section{Proof of Theorem \ref{th:LATbasis}}\label{sec:proof}

To prove Theorem \ref{th:LATbasis}, we will construct a linear map $\phi \colon \Alg(\mathcal{C}) \rightarrow \Alg(\mathcal{C})$ such that $\ker(\phi) = \I$ and $\Img(\phi)= \spn\{\mathcal{B}\}$. The map will be defined explicitly (or iteratively) on the basis $\mathcal{D}$.

Let $\pi_D$ be a linear projection defined by
\begin{align*}
 \pi_D \colon \ \Alg(\mathcal{C}) &\longrightarrow \Alg(\mathcal{C}), \\
 c & \longmapsto c \\ 
 [x_1, x_2 ,x_3] &\longmapsto [\pi_D(x_1),\pi_D(x_2),\pi_D(x_3)], \\
 \mathfrak{s}(x_1 \cdots x_n) \rhd c &\longmapsto \mathfrak{s}(\pi_D(x_1) \cdots \pi_D(x_n)) \rhd c , \\
 \mathfrak{s}(x_1 \cdots x_n) \rhd [y_1, y_2, y_3] &\longmapsto \pi_D \biggl( \sum_{I,J,K} [\mathfrak{s}\bigl(x_{i_1} \cdots x_{i_{|I|}}\bigr) \rhd y_1, \mathfrak{s}(x_{j_1} \cdots x_{j_{|J|}}) \rhd y_2,\\
 & \hphantom{\longmapsto \pi_D \biggl( \sum_{I,J,K}} \mathfrak{s}(x_{k_1} \cdots x_{k_{|K|}}) \rhd y_3] \biggr),
\end{align*}
where $I$, $J$, $K$ are pairwise disjoint subsets of $\{1, \dots , n\}$ such that $I \cup J \cup K = \{1, \dots , n \}$, and for any $s < t$ we have $i_s < i_t$, $j_s < j_t$ and $k_s < k_t$. Observe that $\Img(\pi_D)$ contains no elements on the form $\mathfrak{s}(x_1 \cdots x_n) \rhd [y_1,y_2,y_3]$. Define a set
\[
\Bar{\mathcal{D}} = \{ x \in \mathcal{D} \mid x = \mathfrak{s}(x_1 \cdots x_n) \rhd c , \, x_i \in \mathcal{D} \} \subset \mathcal{D}
\]
and note that $(\Img(\pi_D), [\noarg,\noarg,\noarg]) \subset \Alg_3(\Bar{\mathcal{D}})$, hence we may define a Hall map $\mathfrak{H}$ as in Definition~\ref{def:Hall_map} and restrict it to the image of $\pi_D$.

Let $\pi_H$ be the projections defined by
\begin{align*}
 \pi_H\colon \ \Img(\pi_D) &\longrightarrow \Alg(\mathcal{C}), \\
 c & \longmapsto c, \\
 [x_1, x_2 ,x_3] &\longmapsto \mathfrak{H}([\pi_H(x_1), \pi_H(x_2) , \pi_H(x_3)] ), \\
 \mathfrak{s}(x_1 \cdots x_n) \rhd c & \longmapsto \mathfrak{s}(\pi_H(x_1) \cdots \pi_H(x_n))\rhd c.
\end{align*}

\begin{Note}
 The notation $\pi_D$ and $\pi_H$ is chosen for the following reasons: Firstly, they are both linear projections, which we indicate by the letter $\pi$. The projection $\pi_D$ is constructed to account for the relation \eqref{eq:LAT2}, which is a derivation property by the triangle product on the triple-bracket, hence the letter $D$. In particular, this choice is not connected to the set $\mathcal{D}$. The letter $H$ in $\pi_H$ is there to indicate its relation with Hall bases.
\end{Note}

\begin{Proposition}\label{id-D_in_I}
For any $x \in \Alg(\mathcal{C})$, we have $x-\pi_D(x) \in \mathcal{I}$.
\end{Proposition}

By Lemma \ref{Lemma_D_property}, we know that
\begin{gather*}
 (x_1 \cdots x_n) \rhd [r_1, r_2, r_3] \\
\qquad - \sum_{I,J,K} \bigl[\bigl(x_{i_1} \cdots x_{i_{|I|}}\bigr) \rhd r_1, \bigl(x_{j_1} \cdots x_{j_{|J|}}\bigr) \rhd r_2, \bigl(x_{k_1} \cdots x_{k_{|K|}}\bigr) \rhd r_3\bigr] \in \mathcal{I},
\end{gather*}
where $I$, $J$, $K$ are pairwise disjoint subsets of $\{1, \dots , n\}$ such that $I \cup J \cup K = \{1, \dots , n \}$, and for any $s < t$ we have $i_s < i_t$, $j_s < j_t$ and $k_s < k_t$. In this proof we will always let $I$, $J$ and $K$ denote a triplet of sets with these properties.

\begin{proof} By induction on $|x|_V$: If $|x|_V=1$, we have $x \in \mathcal{C}$ and $x-\pi_D(x) = x-x =0 \in \mathcal{I}$. Assume $x-\pi_D(x) \in \mathcal{I}$ when $|x|_V <k$, then consider the case $|x|_V=k$. By linearity of $\pi_D$, we may assume $x \in \mathcal{D}$. If $x = \mathfrak{s}(x_1 \cdots x_n) \rhd c$, we have
\[
x - \pi_D(x) = \sum_{i=1}^n \mathfrak{s}(\pi_D(x_1) \cdots (x_i-\pi_D(x_i)) \cdots x_n) \rhd c.
\]
Since $|x_i|_V < k$, we have $x_i - \pi_D(x_i) \in \mathcal{I}$, hence also $x-\pi_D(x) \in \mathcal{I}$. If $x = \mathfrak{s}(x_1 \cdots x_n) \rhd [r_1, r_2, r_3]$, we have
\begin{gather*}
 x - \pi_D(x) = \mathfrak{s}(x_1 \cdots x_n) \rhd [r_1, r_2, r_3] \\ \hphantom{x - \pi_D(x) =}{}
 - \pi_D \biggl( \sum_{I,J,K} \bigl[\mathfrak{s}\bigl(x_{i_1} \cdots x_{i_{|I|}}\bigr) \rhd r_1, \mathfrak{s}\bigl(x_{j_1} \cdots x_{j_{|J|}}\bigr) \rhd r_2, \mathfrak{s}\bigl(x_{k_1} \cdots x_{k_{|K|}}\bigr) \rhd r_3\bigr] \biggr) \\ \hphantom{x - \pi_D(x) }{}
 = \mathfrak{s}(x_1 \cdots x_n) \rhd [r_1, r_2, r_3] \\ \hphantom{x - \pi_D(x) =}{}
 - \sum_{I,J,K} \bigl[\mathfrak{s}\bigl(x_{i_1} \cdots x_{i_{|I|}}\bigr) \rhd r_1, \mathfrak{s}\bigl(x_{j_1} \cdots x_{j_{|J|}}\bigr) \rhd r_2, \mathfrak{s}\bigl(x_{k_1} \cdots x_{k_{|K|}}\bigr) \rhd r_3\bigr] \\ \hphantom{x - \pi_D(x) =}{}
 + \sum_{I,J,K} \bigl[\mathfrak{s}\bigl(x_{i_1} \cdots x_{i_{|I|}}\bigr) \rhd r_1, \mathfrak{s}\bigl(x_{j_1} \cdots x_{j_{|J|}}\bigr) \rhd r_2, \mathfrak{s}\bigl(x_{k_1} \cdots x_{k_{|K|}}\bigr) \rhd r_3\bigr] \\ \hphantom{x - \pi_D(x) =}{}
 - \pi_D \biggl( \sum_{I,J,K} [\mathfrak{s}\bigl(x_{i_1} \cdots x_{i_{|I|}}\bigr) \rhd r_1, \mathfrak{s}\bigl(x_{j_1} \cdots x_{j_{|J|}}\bigr) \rhd r_2, \mathfrak{s}\bigl(x_{k_1} \cdots x_{k_{|K|}}\bigr) \rhd r_3] \biggr) \\ \hphantom{x - \pi_D(x)}{}
 = \alpha + \sum_{I,J,K} [\beta_I - \pi_D(\beta_I),\beta_J, \beta_K] + [\pi_D(\beta_I), \beta_J - \pi_D(\beta_J), \beta_K] \\ \hphantom{x - \pi_D(x) =\alpha + \sum_{I,J,K}}{}
 + [\pi_D(\beta_I), \pi_D(\beta_J), \beta_K -\pi_D(\beta_K)],
\end{gather*}
where
\begin{gather*}
\alpha = \mathfrak{s}(x_1 \cdots x_n) \rhd [r_1, r_2, r_3] \\ \hphantom{\alpha =}{}
- \sum_{I,J,K} \bigl[ \mathfrak{s}\bigl(x_{i_1} \cdots x_{i_{|I|}}\bigr) \rhd r_1, \mathfrak{s}\bigl(x_{j_1} \cdots x_{j_{|J|}}\bigr) \rhd r_2, \mathfrak{s}\bigl(x_{k_1} \cdots x_{k_{|K|}}\bigr) \rhd r_3\bigr] \in \mathcal{I}
\end{gather*}
by Lemma \ref{Lemma_D_property} and
\begin{gather*}
 \beta_I = \mathfrak{s}\bigl(x_{i_1} \cdots x_{i_{|I|}}\bigr) \rhd r_1, \qquad
 \beta_J = \mathfrak{s}\bigl(x_{j_1} \cdots x_{j_{|J|}}\bigr)\rhd r_2, \qquad
 \beta_K = \mathfrak{s}\bigl(x_{k_1} \cdots x_{k_{|K|}}\bigr) \rhd r_3.
\end{gather*}
Since $|\beta_t|_V < p$, we have $\beta_t - \pi_D(\beta_t) \in \mathcal{I}$ for $t = I,J,K$, hence $x-\pi_D(x) \in \mathcal{I}$.
\end{proof}

\begin{Proposition}\label{id-h_in_I}
If $x \in \Img(\pi_D)$, then $x-\pi_H(x) \in \mathcal{I}$.
\end{Proposition}

\begin{proof}
If $|x|_V = 1$, we have $x-\pi_H(x) = 0 \in \mathcal{I}$ by definition, since $x \in \mathcal{C}$. Assume the proposition holds for all $x$ with $|x|_V<k$, then consider the case $|x|_V=k$. 
If $x = \mathfrak{s}(x_1 \cdots x_n) \rhd c$, we have
\[
x-\pi_H(x) = \sum_{i=1}^n \mathfrak{s}(\pi_H(x_1) \cdots (x_i - \pi_H(x_i)) \cdots x_n) \rhd c,
\]
where $|x_i|_V <k$, hence $x_i - \pi_H(x_i) \in \mathcal{I}$ by assumption, and we get $x-\pi_H(x) \in \mathcal{I}$. If $x = [x_1 ,x_2, x_3]$, we have
\begin{gather*}
 x- \pi_H(x) = [x_1, x_2, x_3] - \mathfrak{H}([\pi_H(x_1), \pi_H(x_2), \pi_H(x_3)]) \\ \hphantom{x- \pi_H(x)}{}
 = [x_1, x_2, x_3] - [\pi_H(x_1), \pi_H(x_2), \pi_H(x_3)] + [\pi_H(x_1), \pi_H(x_2), \pi_H(x_3)]\\ \hphantom{x- \pi_H(x)=}{}
 - \mathfrak{H}([\pi_H(x_1), \pi_H(x_2), \pi_H(x_3)]) \\ \hphantom{x- \pi_H(x)}{}
 = [x_1 - \pi_H(x_1), x_2 , x_3] + [\pi_H(x_1), x_2 - \pi_H(x_2), x_3]\\ \hphantom{x- \pi_H(x)=}{}
 + [\pi_H(x_1), \pi_H(x_2), x_3 - \pi_H(x_3)]+ ( [\pi_H(x_1), \pi_H(x_2), \pi_H(x_3)] \\ \hphantom{x- \pi_H(x)=}{}
 - \mathfrak{H}([\pi_H(x_1), \pi_H(x_2), \pi_H(x_3)]) ).
\end{gather*}
Since $x_i - \pi_H(x_i) \in \mathcal{I}$ by assumption and
\[
[\pi_H(x_1), \pi_H(x_2), \pi_H(x_3)] - \mathfrak{H}([\pi_H(x_1), \pi_H(x_2), \pi_H(x_3)]) \in \mathcal{I}
\]
 by the definition of $\mathfrak{H}$, we conclude that $x-\pi_H(x) \in \mathcal{I}$. The proposition follows by induction.
\end{proof}

Let $\phi := \pi_H \circ \pi_D$.

\begin{Lemma}\label{lemma:ker}
The ideal $\mathcal{I}$ is the kernel of $\phi$.
\end{Lemma}

\begin{proof}
\textit{Step $1$.} It is clear by the definition of $\pi_D$ that $\phi(x)$ can never contain a term or a~factor of a term on the form
\[
(y_1 \cdots y_n) \rhd [r_1, r_2, r_3]
\]
and hence $\phi(x)$ can never contain a factor satisfying
\[
u \rhd [v,w,z] - [u\rhd v ,w,z] - [v, u \rhd w , z] - [v,w,u \rhd z].
\]
It is also clear by the definition of $\pi_H$ that $\phi(x)$ can not contain a non-Hall triple-bracket as a~term or a factor of a term, hence $\phi(x)$ can never contain a factor of the form
\[
[u,v,w] + [v,w,u] + [w,u,v].
\]
We conclude that either $\phi(x)=0$ or $\phi(x) \notin \mathcal{I}$.

\textit{Step $2$.} Let $x \in \mathcal{I}$ and assume $\phi(x) \neq 0$. By Propositions \ref{id-D_in_I} and \ref{id-h_in_I}, we know that $x-\phi(x)\in \mathcal{I}$, but by step 1 we have $\phi(x) \notin \mathcal{I}$. This is a contradiction and we conclude that $\phi(x)=0$, which proves that $\I \subset \ker(\phi)$.

\textit{Step $3$.} Now, let $x$ be any element in $\ker(\phi)$. By Propositions \ref{id-D_in_I} and \ref{id-h_in_I}, we know that $x- \phi(x) \in \I$, but $\phi(x)=0$, hence $x\in \I$ and $\I \subset \ker(\phi)$.
\end{proof}

\begin{Lemma}\label{lemma:Im}
The image of $\phi$ is equal to the span of $\mathcal{B}$.
\end{Lemma}

\begin{proof}
By induction on $|\noarg |_V$. 
If $x \in \mathcal{B}$ with $|x|_V=1$, then $\phi(x) = x$ since $\phi(c)=c$ for all elements $c \in \mathcal{C}$. 
\begin{enumerate}\itemsep=0pt
 \item[(i)] If $x = \mathfrak{s}(x_1 \cdots x_n) \rhd c$, we have by induction
 \[
 \phi(x) = \mathfrak{s}(\phi(x_1) \cdots \phi(x_n))\rhd c = \mathfrak{s}(x_1 \cdots x_n) \rhd c = x.
 \]
 \item[(ii)] If $x = [x_1, x_2, x_3]$, we have by induction
 \[
 \phi(x) = [\phi(x_1),\phi(x_2),\phi(x_3)] = [x_1, x_2, x_3] = x.
 \]
\end{enumerate}
Hence we get $\phi(x)=x$ for all $x \in \mathcal{B}$ which proves that $\mathcal{B} \subset \Img(\phi)$.
Notice that for $|x|_V=1$ we have $\phi(x) \in \spn\{\mathcal{B} \}$. Assume $\phi(y) \in \spn \{\mathcal{B}\}$ for all $y\in \mathcal{D}$ with $|y|_V < k$ and let $x \in \mathcal{D}$ with $|x|_V=k$. If $x = \mathfrak{s}(x_1 \cdots x_n) \rhd c$, we get $\phi(x) = \mathfrak{s}( \phi(x_1) \cdots \phi(x_n)) \rhd c$ and since $|x_i|_V < |x|_V$, we have $\phi(x_i) \in \spn\{ \mathcal{B} \}$ by assumption, and then $\phi(x)$ must be in the span of $\mathcal{B}$.

If $x = \mathfrak{s}(x_1 \cdots x_n) \rhd [y_1, y_2 ,y_3]$, we have
\[
\phi(x) = \sum_{i \in I} \mathfrak{H}([\phi(\alpha_i), \phi(\beta_i), \phi(\gamma_i)]),
\]
where $\alpha_i$, $\beta_i$ and $\gamma_i$ are as described in the definition of $\pi_D$ and $i\in I$ is any indexing of these elements. Since $|\alpha_i|_V<|x|_V$, $|\beta_i|_V<|x|_V$ and $|\gamma_i|_V<|x|_V$ for all $i \in I$, we know that $\phi(\alpha_i) \in \spn \{ \mathcal{B} \}$, $\phi(\beta_i) \in \spn \{ \mathcal{B} \}$ and $\phi(\gamma_i) \in \spn \{ \mathcal{B} \}$ for all $i \in I$, and furthermore, by the definition of $\mathfrak{H}$ we know that $\mathfrak{H}([\phi(\alpha_i), \phi(\beta_i), \phi(\gamma_i)])$ will be a Hall Lie triple element of elements from $\mathcal{B}$, hence it is itself in $\mathcal{B}$.
This means that $\phi(x) \in \spn \{\mathcal{B}\}$ for all $x \in \mathcal{D}$, and since $\mathcal{D}$ is a~basis of $\Alg(\mathcal{C})$, we know that $\Img(\phi) \subset \spn \{\mathcal{B}\}$.
\end{proof}

\begin{proof}[Proof of Theorem~\ref{th:LATbasis}]
Since $\phi$ is a linear map, we have
\[
\Img(\phi) \cong \Alg(\mathcal{C})/\ker(\phi).
\]
By Lemma \ref{lemma:Im}, $\Img(\phi) = \spn \{\mathcal{B} \}$, and by Lemma \ref{lemma:ker}, $\ker(\phi) = \I$. By definition, $\LAT(\mathcal{C}) \cong \Alg(\mathcal{C})/ \I$ and we have
\[
\LAT(\mathcal{C}) \cong \spn \{ \mathcal{B} \}.
\]
Furthermore, by Corollary \ref{cor:B_lin_ind} we know that the set $\mathcal{B}$ is linearly independent, hence $\mathcal{B}$ is a basis for the free Lie admissible triple algebra.
\end{proof}

\section{Embedding of LAT into post-Lie}\label{sec:PL-emb}

Let $\mathcal{L}$ be a Lie algebra with an involutive automorphism $\sigma \colon \mathcal{L} \rightarrow \mathcal{L}$, that is an isomorphism satisfying $\sigma^2(x) = x$ for all $x \in \mathcal{L}$. Then $\mathcal{L}$ decomposes into $\mathcal{L} = \mathcal{L}_- \oplus \mathcal{L}_+$ where $\mathcal{L}_- := \{ x \in \mathcal{L} \mid \sigma(x) = -x \}$ and $\mathcal{L}_+ := \{ x \in \mathcal{L} \mid \sigma(x) = x \}$. We observe that
\[
[\mathcal{L}_+ , \mathcal{L}_+] \subset \mathcal{L}_+, \qquad [\mathcal{L}_- , \mathcal{L}_+] \subset \mathcal{L}_-, \qquad [\mathcal{L}_- , \mathcal{L}_- ] \subset \mathcal{L}_+.
\]
In other words, $\mathcal{L}$ is a $\mathbb{Z}_2$-graded Lie algebra, and we may define a triple-bracket
\[
[\noarg, \noarg, \noarg ] \colon \ \mathcal{L}_- \times \mathcal{L}_- \times \mathcal{L}_- \rightarrow \mathcal{L}_-
\]
by $[x,y,z] := [[x,y],z]$. This makes $\mathcal{L}_-$ into a Lie triple system, defined in Lemma \ref{lemma:LTS}. The following proposition is originally due to Jacobson \cite{Jacobson1951}, and improved by Yamaguti \cite{yamaguti1957}.

\begin{Proposition}\label{prop:LTSembedding}
Let $T$ be a Lie triple system. For $x,y \in T$ let $D_{x,y}$ be the linear transformation $D_{x,y}(z) = [x,y,z]$. Let $\Der(T)$ be the subspace of all linear transformations of $T$ spanned by $D_{x,y}$, and let $\mathcal{L}$ be the direct sum vector space $\Der(T) \oplus T$. Then $\mathcal{L}$ becomes a $\mathbb{Z}_2$-graded Lie algebra with the bracket
\[
[(A,x),(B,y)] = (A\circ B - B \circ A + D_{x,y} , A(y) - B(x)).
\]
The map $\sigma$ defined by $\sigma((A,x)) = (A,-x)$ is an involutive automorphism.
\end{Proposition}

\begin{proof}
See \cite[Proposition 2.3]{Loos1}.
\end{proof}

\begin{Remark}\label{rmrk:82}
Let $\mathcal{L} = \Der(T) \oplus T$ be as in the proposition above, but define the bracket by
\[
[(A,x),(B,y)] = (-A\circ B + B \circ A + D_{x,y}, B(x) - A(y)).
\]
This will also define a Lie algebra, and $\sigma\colon (A,x) \mapsto (A,-x)$ is an involutive automorphism of this Lie algebra as well.
\end{Remark}

It is natural to ask how this theorem extends to Lie admissible triple algebras. The relation between LAT and $\mathbb{Z}_2$-graded Lie algebras have been examined in \cite{GS_00,Sokolov_17}. This embedding can be further specialized by demanding that the $\mathbb{Z}_2$-graded Lie algebra has a post-Lie structure, and that the products of the Lie admissible triple algebra agrees with the post-Lie product.

\begin{Proposition}\label{prop:Lie_adm_in_psot-Lie}
Let $(A,[\noarg,\noarg],\blacktriangleright)$ be a post-Lie algebra. Define a product $\rhd$ by
\[
x \rhd y = x \blacktriangleright y + \frac{1}{2}[x,y],
\]
then $(A,\rhd)$ is Lie admissible. The skew associator is
\[
\operatorname{ass}_{\rhd}(x,y,z) - \operatorname{ass}_{\rhd}(y,x,z) = - \frac{1}{4}[[x,y],z].
\]
\end{Proposition}

\begin{proof}
See \cite[Proposition 2.7]{MK_Lundervold_2013}.
\end{proof}

\begin{Proposition}
Let $(A,[\noarg,\noarg],\blacktriangleright)$ be a post-Lie algebra. Define a product $\rhd$ by
\[
x \rhd y = x \blacktriangleright y + \frac{1}{2}[x,y],
\]
then $(A,\rhd)$ is a Lie admissible triple algebra.
\end{Proposition}

\begin{proof}
Define
\[
[x,y,z] := \operatorname{ass}_{\rhd}(x,y,z) - \operatorname{ass}_{\rhd}(y,x,z).
\]
By Proposition \ref{prop:Lie_adm_in_psot-Lie} above, we know that
\[
[x,y,z]+[y,z,x] + [z,x,y] = -\frac{1}{4}([[x,y],z]+[[y,z],x] + [[z,x],y] ) = 0.
\]
We need to show that $\rhd$ acts as a derivation on $[\noarg,\noarg,\noarg]$,
\begin{gather*}
 w \rhd [x,y,z] = w \blacktriangleright \biggl(-\frac{1}{4}[[x,y],z]\biggr) + \frac{1}{2}[w , -\frac{1}{4}[[x,y],z]] \\ \hphantom{w \rhd [x,y,z] }{}
 = -\frac{1}{4} ([w \blacktriangleright [x,y], z] + [[x,y],w\blacktriangleright z] ) +\frac{1}{2} \biggl(-\frac{1}{4} ([[w,[x,y]],z] + [[x,y],[w,z]] ) \biggr) \\ \hphantom{w \rhd [x,y,z] }{}
 = -\frac{1}{4} \biggl( [[w \blacktriangleright x,y], z] + [[x, w\blacktriangleright y],z] + [[x,y],w\blacktriangleright z] \\ \hphantom{w \rhd [x,y,z] =-\frac{1}{4}}{}
 +\frac{1}{2} ([[[w,x],y],z] + [[x,[w,y]],z] + [[x,y],[w,z]] ) \biggr) \\ \hphantom{w \rhd [x,y,z] }{}
 = -\frac{1}{4} \biggl([[w \blacktriangleright x,y], z] + \frac{1}{2}[[[w,x],y],z] \biggr) - \frac{1}{4} \biggl([[x, w\blacktriangleright y],z] +\frac{1}{2}[[x,[w,y]],z] \biggr) \\ \hphantom{w \rhd [x,y,z] =}{}
 - \frac{1}{4} \biggl( [[x,y],w\blacktriangleright z]+\frac{1}{2}[[x,y],[w,z]] \biggr) \\ \hphantom{w \rhd [x,y,z] }{}
 = [w\rhd x ,y,z] + [x,w\rhd y,z] + [x,y,w\rhd z].
\tag*{\qed}
\end{gather*}
\renewcommand{\qed}{}
\end{proof}

\begin{Example}
Let $(\mathcal{A}, \rhd)$ be a LAT-algebra, hence it is also a Lie admissible algebra. Define $[x,y]=x\rhd y - y \rhd x$ for all $x,y \in \mathcal{A}$, then $(\mathcal{A},[\noarg,\noarg])$ is a Lie algebra. For any Lie algebra if we define $x \blacktriangleright y := [y,x]$, we get a post-Lie algebra, hence $(\mathcal{A},[\noarg,\noarg],\blacktriangleright)$ is a post-Lie algebra where $x\blacktriangleright y = y \rhd x - x \rhd y$.
\end{Example}

This is not the embedding we are interested in, as this is not related to the Lie triple embedding in Proposition \ref{prop:LTSembedding}. What we want is to embed the Lie triple algebra $(\mathcal{A},[\noarg,\noarg,\noarg])$ into a Lie algebra $\mathcal{L}=\Der(\mathcal{A})\oplus \mathcal{A}$, and then extend the Lie admissible triple operator $\rhd$ to an operator~$\blacktriangleright$ on~$\mathcal{L}$ in such a way that
\begin{enumerate}\itemsep=0pt
 \item[(i)] $(\mathcal{L},[\noarg,\noarg],\blacktriangleright)$ is a post-Lie algebra,
 \item[(ii)] and the post-Lie operator reduces to the LAT operator when restricted to $\mathcal{A}$, i.e.,
 \[
 \blacktriangleright \big|_{\mathcal{A}} = \rhd.
 \]
\end{enumerate}
If we let $(\mathcal{L},[\noarg,\noarg])$ be the standard Lie algebra embedding of the Lie triple system $(\mathcal{A},[\noarg,\noarg,\noarg])$, where $(\mathcal{A},\rhd)$ is a LAT algebra. 
Using the post-Lie relations, we see that the extension must satisfy \looseness=-1
\begin{gather*}
 x \blacktriangleright y = x \rhd y, \\
 D_{x,y} \blacktriangleright z = [x,y] \blacktriangleright z = [x,y,z] = D_{x,y} (z), \\
 x \blacktriangleright D_{y,z} = x \blacktriangleright [y,z] = [x\blacktriangleright y,z] + [y,x \blacktriangleright z] = D_{x\blacktriangleright y,z} + D_{y, x \blacktriangleright z} ,\\
 D_{x,y} \blacktriangleright D_{u,v} = [x,y] \blacktriangleright [u,v] = [[x,y,u],v] + [u,[x,y,v]] = D_{[x,y,u],v} + D_{u,[x,y,v]}.
\end{gather*}

\begin{Theorem}\label{th:pL_emb}
Let $(\mathcal{A},\rhd)$ be a Lie admissible triple algebra with triple-bracket
\[
[x,y,z] = \operatorname{ass}_{\rhd}(x,y,z) -\operatorname{ass}_{\rhd}(y,x,z).
\]
For each pair $x,y \in \mathcal{A}$, define
\begin{align*}
& D_{x,y}\colon \ \mathcal{A} \longrightarrow \mathcal{A}, \qquad
 z \longmapsto [x,y,z]
\end{align*}
and let $\Der(\mathcal{A}) = \{ D_{x,y} \mid x,y \in \mathcal{A} \}$. Let $\mathcal{L}=\Der(\mathcal{A}) \oplus \mathcal{A}$, and let $x,y \in \mathcal{A}$ and $A,B \in \Der(\mathcal{A})$. Define a bracket by
\begin{gather*}
 [x,y] = D_{x,y}, \qquad
 [A,x] = -A(x), \qquad
 [A,B] = -A \circ B + B \circ A.
\end{gather*}
Extend $\rhd$ to all of $\mathcal{L}$ by
\begin{gather*}
 x \blacktriangleright y = x \rhd y, \qquad
 A \blacktriangleright x = A(x), \\
 (x \blacktriangleright A)(y) = x \blacktriangleright A(y) - A(x\blacktriangleright y), \qquad
 A \blacktriangleright B = A \circ B - B \circ A.
\end{gather*}
Then $(\mathcal{L},[\noarg,\noarg],\blacktriangleright)$ is a $\mathbb{Z}_2$-graded post-Lie algebra.
\end{Theorem}

For the sake of convenience, we will write $\operatorname{ass}_{\blacktriangleright}(x,y,z) - \operatorname{ass}_{\blacktriangleright}(y,x,z) = [x,y,z]$.

\begin{proof}
$(\mathcal{L},[\noarg,\noarg])$ is a $\mathbb{Z}_2$-graded Lie algebra by Proposition \ref{prop:LTSembedding} and Remark \ref{rmrk:82}. Throughout this proof, let $x,y,z,w,t \in \mathcal{A}$ and $A,B,C \in \Der(\mathcal{A})$. From the definition of $(x\blacktriangleright A)$, we get
\begin{align*}
 (x \blacktriangleright D_{y,z})(w) & = x \blacktriangleright (D_{y,z}(w)) - D_{y,z}(x\blacktriangleright w)
 = x \blacktriangleright [y,z,w] - [y,z,x \blacktriangleright w] \\
 & = [x \blacktriangleright y, z, w] + [y,x\blacktriangleright z, w]
 = (D_{x\blacktriangleright y,z} + D_{y,x \blacktriangleright z})(w),
\end{align*}
which gives the equality
\begin{equation}\label{x>D}
 x \blacktriangleright D_{y,z} = D_{x\blacktriangleright y,z} + D_{y,x \blacktriangleright z}.
\end{equation}
From the definition of $A\blacktriangleright B$, we get
\begin{align*}
 (D_{x,y}\blacktriangleright D_{z,w})(t) & = (D_{x,y} \circ D_{z,w} - D_{z,w} \circ D_{x,y})(t)
 = [x,y,[z,w,t]] - [z,w,[x,y,t]] \\
 & = [[x,y,z],w,t] + [z,[x,y,w],t]
 = (D_{[x,y,z],w} + D_{z,[x,y,w]})(t),
\end{align*}
which gives the equality
\begin{equation}\label{A>B}
 D_{x,y}\blacktriangleright D_{z,w} = D_{[x,y,z],w} + D_{z,[x,y,w]}.
\end{equation}

We need to verify that the post-Lie relations are satisfied for all combinations of inputs from $\Der(\mathcal{A})$ and $\mathcal{A}$.

(i) Using \eqref{x>D}, we get
 \[
 x \blacktriangleright [y,z] = x \blacktriangleright D_{y,z} = D_{x \blacktriangleright y,z} + D_{y,x \blacktriangleright z} = [x \blacktriangleright y,z ] + [y, x\blacktriangleright z].
 \]

(ii) By the definition of $x \blacktriangleright A$, we get
 \[
 x \blacktriangleright [A,y] = - x \blacktriangleright(A(y)) = -(x\blacktriangleright A)(y) - A(x\blacktriangleright y) = [x \blacktriangleright A ,y] + [A, x \blacktriangleright y].
 \]

(iii) Follows from skew symmetry and the relation above
 \[
 x \blacktriangleright [y,A] = -x \blacktriangleright [A,y] = - [x \blacktriangleright A ,y] - [A, x \blacktriangleright y] = [x \blacktriangleright y,A] + [y, x \blacktriangleright A].
 \]

(iv) Let $B=D_{y,z}$. Using (i) and (ii), we can compute
 \begin{align*}
 x \blacktriangleright [A,B] ={}& x \blacktriangleright [A,[y,z]] = x \blacktriangleright ([[A,y],z] + [y,[A,z]] ) \\
 ={}& [x \blacktriangleright [A,y],z] + [[A,y], x \blacktriangleright z] +[x \blacktriangleright y, [A,z]] + [y, x \blacktriangleright [A,z]] \\
 ={}& [[x \blacktriangleright A, y],z] + [[A,x\blacktriangleright y],z] + [[A,y],x \blacktriangleright z] + [x \blacktriangleright y, [A,z]] \\
 &{}+ [y,[x\blacktriangleright A, z]] + [y,[A,x\blacktriangleright z]] \\
 ={}& [x\blacktriangleright A, [y,z]] + [A,[x\blacktriangleright y,z]] + [A,[y,x\blacktriangleright z]] \\
 ={}& [x\blacktriangleright A,B] + [A, x\blacktriangleright B].
 \end{align*}

(v) Let $A= D_{z,w}$, then by \eqref{A>B}, we get
 \begin{align*}
& A \blacktriangleright [x,y] = D_{z,w} \blacktriangleright D_{x,y}
 = D_{[z,w,x],y} + D_{x,[z,w,y]}
 = [A \blacktriangleright x, y] + [x, A \blacktriangleright y].
 \end{align*}

(vi) Let $A = D_{w,t}$ and $B = D_{y,z}$, then
 \begin{align*}
 A\blacktriangleright [B,x] & = - [w,t ,[y,z,x]]
 = -[[w,t,y],z,x] - [y,[w,t,z],x] - [y,z,[w,t,x]] \\
 & = [A \blacktriangleright B,x] + [B, A\blacktriangleright x].
 \end{align*}

(vii) From skew symmetry and the relation above gives \[
 A \blacktriangleright [x,B] = -A \blacktriangleright [B,x] = -[A \blacktriangleright B,x] - [B, A\blacktriangleright x] = [A \blacktriangleright x,B] + [x, A \blacktriangleright B].
 \]

(viii) $ A \blacktriangleright [B,C] = -[A,[B,C]] = -[[A,B],C] - [B,[A,C]] = [A\blacktriangleright B,C] + [B, A\blacktriangleright C]$.

(ix) $[x,y]\blacktriangleright z = D_{x,y} (z) = [x,y,z]$.

(x) From the definition of $(x \blacktriangleright A)(y)$, we have $(x \blacktriangleright A)(y) - x \blacktriangleright A(y) + A(x\blacktriangleright y) = 0$, hence we get
 \begin{align*}
 [A,x] \blacktriangleright y & = -A(x) \blacktriangleright y + ( (x \blacktriangleright A)(y) - x \blacktriangleright A(y) + A(x\blacktriangleright y) ) \\
 & = - (A\blacktriangleright x) \blacktriangleright y + (x \blacktriangleright A) \blacktriangleright y - x \blacktriangleright ( A \blacktriangleright y) + A \blacktriangleright (x \blacktriangleright y)
 = [A,x,y].
 \end{align*}

(xi) By skew symmetry and the relation above, we get
 \[
 [x,A]\blacktriangleright y = - [A,x] \blacktriangleright y = -[A,x,y] = [x,A,y].
 \]

(xii) By the Jacobi identity, we have $[B,[A,x]] - [[B,A],x] - [A,[B,x]]=0$, hence
 \begin{align*}
 [A,B] \blacktriangleright x & = [A,B]\blacktriangleright x - ([B,[A,x]] - [[B,A],x] - [A,[B,x]] ) \\
 & = [A,[B,x]] + [A,B]\blacktriangleright x - [B,[A,x]] + [[B,A],x] \\
 & = A \blacktriangleright (B \blacktriangleright x ) - (A \blacktriangleright B) \blacktriangleright x - B \blacktriangleright (A \blacktriangleright x) + (B \blacktriangleright A) \blacktriangleright x
 = [A,B,x].
 \end{align*}

(xiii) Let $A = D_{z,w}$, then by repeatedly using \eqref{x>D}, we get
 \begin{align*}
 [x,y,A] ={}& x \blacktriangleright (y \blacktriangleright D_{z,w}) - (x \blacktriangleright y) \blacktriangleright D_{z,w} - y \blacktriangleright (x \blacktriangleright D_{z,w}) + (y \blacktriangleright x) \blacktriangleright D_{z,w} \\
 ={}& x\blacktriangleright (D_{y\blacktriangleright z,w} + D_{z,y \blacktriangleright w} ) - (D_{(x\blacktriangleright y) \blacktriangleright z,w} + D_{z, (x\blacktriangleright y) \blacktriangleright w} ) \\
 &{} - y \blacktriangleright (D_{x\blacktriangleright z,w} + D_{z,x \blacktriangleright w} ) + (D_{(y\blacktriangleright x) \blacktriangleright z,w} + D_{z, (y\blacktriangleright x) \blacktriangleright w} ) \\
 ={}& D_{x \blacktriangleright (y \blacktriangleright z),w} - D_{(x\blacktriangleright y) \blacktriangleright z,w} - D_{y \blacktriangleright (x\blacktriangleright z),w} + D_{(y\blacktriangleright x) \blacktriangleright z,w} \\
 &{}+ D_{z, x \blacktriangleright (y \blacktriangleright w)} - D_{z, (x\blacktriangleright y) \blacktriangleright w} - D_{z,y \blacktriangleright (x \blacktriangleright w)} + D_{z, (y\blacktriangleright x) \blacktriangleright w} \\
 &{}+ D_{y\blacktriangleright z, x\blacktriangleright w} + D_{x \blacktriangleright z, y \blacktriangleright w} - D_{x \blacktriangleright z, y \blacktriangleright w} - D_{y \blacktriangleright z, x \blacktriangleright w} \\
 ={}& D_{[x,y,z],w} + D_{z,[x,y,w]}.
 \end{align*}
 By \eqref{A>B}, we have
 \[
 [x,y] \blacktriangleright A = D_{x,y} \blacktriangleright D_{z,w} = D_{[x,y,z],w} + D_{z,[x,y,w]} = [x,y,A].
 \]

(xiv) From (iv), we have $x \blacktriangleright [A,B] - [x \blacktriangleright A,B] - [A, x\blacktriangleright B] = 0$, hence we get
 \begin{align*}
 [A,x] \blacktriangleright B & = [A, x] \blacktriangleright B + ( x \blacktriangleright [A,B] - [x \blacktriangleright A,B] - [A, x\blacktriangleright B] ) \\
 & = - (A \blacktriangleright x) \blacktriangleright B - x \blacktriangleright (A \blacktriangleright B) + (x \blacktriangleright A) \blacktriangleright B + A \blacktriangleright (x \blacktriangleright B )
 = [A,x,B].
 \end{align*}

(xv) By skew symmetry and the relation above, we get
 \[
 [x,A] \blacktriangleright B = -[A,x] \blacktriangleright B = -[A,x,B] = [x,A,B].
 \]

(xvi) The Jacobi identity gives $[B,[A,C]] - [[B,A],C] - [A,[B,C]]=0$, and we get
 \begin{align*}
 [A,B] \blacktriangleright C &{}= -[[A,B],C] - ( [B,[A,C]] - [[B,A],C] - [A,[B,C]] ) \\
 &{}= - (A\blacktriangleright B ) \blacktriangleright C - B \blacktriangleright ( A \blacktriangleright C) + (B \blacktriangleright A ) \blacktriangleright C + A \blacktriangleright (B \blacktriangleright C) \\
 &{}= [A,B,C].
\tag*{\qed}
\end{align*}
\renewcommand{\qed}{}
\end{proof}

\begin{Remark}
 There is a way to get a post-Lie structure associated to a homogeneous space by considering a sub-bundle of the tangent bundle of the frame bundle \cite{grong_MK_S_2023}. The construction of a post-Lie algebra on this bundle is very similar to the construction of a post-Lie algebra on $\Der(\mathcal{A}) \oplus \mathcal{A}$ presented in Theorem \ref{th:pL_emb}.
\end{Remark}

For a Lie admissible triple algebra $\mathcal{A}$, let $\mathfrak{g}(\mathcal{A})$ be the associated $\mathbb{Z}_2$-graded post-Lie algebra of~$\mathcal{A}$.

\begin{Corollary}
 Let $\mathcal{B}$ be the basis of $\LAT(\mathcal{C})$ given in Theorem~$\ref{th:LATbasis}$. Let $\hat{\mathcal{B}}$ be defined by
 \begin{enumerate}\itemsep=0pt
 \item[$(i)$] if $x\in \mathcal{B}$ then $x \in \hat{\mathcal{B}}$, and
 \item[$(ii)$] if $x,y \in \mathcal{B}$ with $x>y$, then $[x,y] \in \hat{\mathcal{B}}$.
 \end{enumerate}
 Then $\hat{\mathcal{B}}$ is a basis for $\mathfrak{g}(\LAT(\mathcal{C}))$.
\end{Corollary}

\begin{proof}
 By Theorem \ref{th:pL_emb}, $\mathfrak{g}(\LAT(\mathcal{C})) = \LAT(\mathcal{C}) \oplus \Der(\LAT(\mathcal{C}))$ where $\Der(\LAT(\mathcal{C})) = \{ [x,y] \mid x,y \in \LAT(\mathcal{C}) \}$. By Theorem \ref{th:LATbasis}, we know that $\mathcal{B}$ is a basis for $\LAT(\mathcal{C})$, and by definition of $\Der(\LAT(\mathcal{C}))$ we get that $\{[x,y] \mid x,y \in \mathcal{B}, \, x>y \}$ is a basis of $\Der(\LAT(\mathcal{C}))$.
\end{proof}

\appendix

\section{Appendix}

\begin{Lemma}\label{lemma:bin_ind}
 For all positive integers $n$ and $k$, the following relation holds:
 \begin{equation}\label{eq:bin_ind}
 \sum_{j=1}^{n-1}(n-j)\binom{(k-2) + (n-j)}{(n-j)} + \binom{k + (n-1)}{(n-1)} = n\binom{(k-1)+(n-1)}{(n-1)}.
 \end{equation}
\end{Lemma}

\begin{proof}
 Let $\beta_k(n)$ be given as the left-hand side of \eqref{eq:bin_ind}. Notice that this equation is trivially true for all $k$ when $n=1$, and assume that is true for all $k$ when $n<N$ for some integer $N$. Notice that
 \begin{align*}
 \beta_k(N-1) &= \sum_{j=1}^{N-2}(N-1-j)\binom{(k-2) + (N-1-j)}{(N-1-j)} + \binom{k + (N-2)}{(N-2)} \\
 &= \sum_{j=2}^{N-1}(N-j)\binom{(k-2) + (N-j)}{(N-j)} + \binom{k + (N-2)}{(N-2)} \\
 &= \beta_k(N) - (N-1)\binom{(k-2)+(N-1)}{(N-1)} - \binom{k + (N-1)}{(N-1)} +\binom{k + (N-2)}{(N-2)}.
 \end{align*}
 Rewriting this equality and using the assumption, we get
 \[
 \beta_k(N) = (N-1)\binom{k +N-3}{N-2} + (N-1)\binom{k +N-3}{N-1} + \binom{k+N-1}{N-1} - \binom{k +N-2}{N-2}.
 \]
 Using Pascal's rule, $\binom{m}{r}= \binom{m-1}{r} + \binom{m-1}{r-1}$, we get
 \[
 \beta_k(N) = (N-1)\binom{k+N-2}{N-1} + \binom{k+N-2}{N-1} = N\binom{(k-1) + (N-1)}{(N-1)}
 \]
 and the lemma follows by induction on $n$.
\end{proof}

\section{Appendix}

Recall
\begin{align*}
 & x \rhd (uv)= (x\rhd u)v + u(x\rhd v), \\
& (xu) \rhd v= x \rhd (u \rhd v) - (x \rhd u) \rhd v
\end{align*}
for $x \in \Alg(\mathcal{C})$ and $u,v \in D(\Alg(\mathcal{C}))$.

\begin{Lemma}\label{Lemma_D_property}
Let $x_0, \dots , x_n , r_1 , r_2, r_3 \in \Alg(\mathcal{C})$. Then
\begin{gather*}
 (x_1 \cdots x_n) \rhd [r_1, r_2, r_3] \\
\qquad {}- \sum_{I,J,K} \bigl[\bigl(x_{i_1} \cdots x_{i_{|I|}}\bigr) \rhd r_1, \bigl(x_{j_1} \cdots x_{j_{|J|}}\bigr) \rhd r_2, \bigl(x_{k_1} \cdots x_{k_{|K|}}\bigr) \rhd r_3\bigr] \in \mathcal{I},
\end{gather*}
where $I$, $J$, $K$ are pairwise disjoint subsets of $\{1, \dots , n\}$ such that $I \cup J \cup K = \{1, \dots , n \}$, and for any $s < t$ we have $i_s < i_t$, $j_s < j_t$ and $k_s < k_t$.
\end{Lemma}

\begin{proof}
If $n=1$, this is true by definition of $\mathcal{I}$. Assume it is true for $n=k$, then we get
\begin{gather*}
 (x_0 x_1 \cdots x_k) \rhd [r_1, r_2, r_3] = x_0 \rhd ( (x_1 \cdots x_k) \rhd [r_1, r_2 , r_3] ) \\ \hphantom{(x_0 x_1 \cdots x_k) \rhd [r_1, r_2, r_3] =}{}
 - \sum_{s=1}^k (x_1 \cdots (x_0 \rhd x_s) \cdots x_k) \rhd [r_1, r_2, r_3],
\end{gather*}
and we know that
\begin{gather*}
 (x_1 \cdots x_k) \rhd [r_1, r_2, r_3] \\
\qquad {}- \sum_{I,J,K} \bigl[\bigl(x_{i_1} \cdots x_{i_{|I|}}\bigr) \rhd r_1, \bigl(x_{j_1} \cdots x_{j_{|J|}}\bigr) \rhd r_2, \bigl(x_{k_1} \cdots x_{k_{|K|}}\bigr) \rhd r_3\bigr] \in \mathcal{I} \\
 (x_1 \cdots (x_0 \rhd x_s) \cdots x_k) \rhd [r_1, r_2, r_3] \\
\qquad {}- \sum_{I,J,K} \bigl[\bigl(\Tilde{x}_{i_1} \cdots \Tilde{x}_{i_{|I|}}\bigr) \rhd r_1, \bigl(\Tilde{x}_{j_1} \cdots \Tilde{x}_{j_{|J|}}\bigr) \rhd r_2, \bigl(\Tilde{x}_{k_1} \cdots \Tilde{x}_{k_{|K|}}\bigr) \rhd r_3\bigr] \in \mathcal{I},
\end{gather*}
where $\Tilde{x}_t=x_t$ whenever $t\neq s$ and $\Tilde{x}_s=x_0 \rhd x_s$. We also have
\begin{gather*}
 \sum_{I,J,K} \bigl( x_0 \rhd [\bigl(x_{i_1} \cdots x_{i_{|I|}}\bigr) \rhd r_1, \bigl(x_{j_1} \cdots x_{j_{|J|}}\bigr) \rhd r_2, \bigl(x_{k_1} \cdots x_{k_{|K|}}\bigr) \rhd r_3] \\ \hphantom{\sum_{I,J,K}\bigl(}{}
 - \bigl[x_0 \rhd \bigl( \bigl(x_{i_1} \cdots x_{i_{|I|}}\bigr) \rhd r_1 \bigr), \bigl(x_{j_1} \cdots x_{j_{|J|}}\bigr) \rhd r_2, \bigl(x_{k_1} \cdots x_{k_{|K|}}\bigr) \rhd r_3\bigr] \\ \hphantom{\sum_{I,J,K}\bigl(}{}
 - \bigl[\bigl(x_{i_1} \cdots x_{i_{|I|}}\bigr) \rhd r_1, x_0 \rhd \bigl( \bigl(x_{j_1} \cdots x_{j_{|J|}}\bigr) \rhd r_2\bigr) , \bigl(x_{k_1} \cdots x_{k_{|K|}}\bigr) \rhd r_3\bigr] \\ \hphantom{\sum_{I,J,K}\bigl(}{}
 - \bigl[\bigl(x_{i_1} \cdots x_{i_{|I|}}\bigr) \rhd r_1, \bigl(x_{j_1} \cdots x_{j_{|J|}}\bigr) \rhd r_2, x_0 \rhd \bigl( \bigl(x_{k_1} \cdots x_{k_{|K|}}\bigr) \rhd r_3 \bigr) \bigr] \bigr) \in \mathcal{I}.
\end{gather*}
Notice that
\[
x_0 \rhd \bigl( \bigl(x_{i_1} \cdots x_{i_{|I|}}\bigr) \rhd r_1 \bigr) = \bigl(x_0x_{i_1} \cdots x_{i_{|I|}}\bigr) \rhd r_1 + \sum_{s=i_1}^{i_{|I|}} \bigl(x_{i_1} \cdots (x_0 \rhd x_s) \cdots x_{i_{|I|}}\bigr) \rhd r_1,
\]
hence we can rewrite the relation above as
\begin{gather*}
 \sum_{I,J,K} \bigl( x_0 \rhd \bigl[\bigl(x_{i_1} \cdots x_{i_{|I|}}\bigr) \rhd r_1, \bigl(x_{j_1} \cdots x_{j_{|J|}}\bigr) \rhd r_2, \bigl(x_{k_1} \cdots x_{k_{|K|}}\bigr) \rhd r_3\bigr] \bigr)\\
 \qquad - \sum_{s=1}^k \sum_{I,J,K} \bigl[\bigl(\Tilde{x}_{i_1} \cdots \Tilde{x}_{i_{|I|}}\bigr) \rhd r_1, \bigl(\Tilde{x}_{j_1} \cdots \Tilde{x}_{j_{|J|}}\bigr) \rhd r_2, \bigl(\Tilde{x}_{k_1} \cdots \Tilde{x}_{k_{|K|}}\bigr) \rhd r_3\bigr] \\
 \qquad - \sum_{\hat{I},\hat{J},\hat{K}} \bigl[\bigl(x_{i_1} \cdots x_{i_{|\hat{I}|}}\bigr) \rhd r_1, \bigl(x_{j_1} \cdots x_{j_{|\hat{J}|}}\bigr) \rhd r_2, \bigl(x_{k_1} \cdots x_{k_{|\hat{K}|}}\bigr) \rhd r_3\bigr] \in \mathcal{I},
\end{gather*}
where $\hat{I}$, $\hat{J}$ and $\hat{K}$ are defined in the same way as $I$, $J$ and $K$, just on the set $\{0, 1, \dots ,k \}$. When put together, this gives
\begin{gather*}
 x_0 \rhd \bigl( (x_1 \cdots x_k) \rhd [r_1, r_2, r_3] \\
 \qquad {}- \sum_{I,J,K} \bigl[\bigl(x_{i_1} \cdots x_{i_{|I|}}\bigr) \rhd r_1, \bigl(x_{j_1} \cdots x_{j_{|J|}}\bigr) \rhd r_2, \bigl(x_{k_1} \cdots x_{k_{|K|}}\bigr) \rhd r_3\bigr] \bigr) \\
 \qquad {}- \sum_{s=1}^k \biggl( (x_1 \cdots (x_0 \rhd x_s) \cdots x_k) \rhd [r_1, r_2, r_3] \\
 \qquad {}\hphantom{- \sum_{s=1}^k \biggl(}{}- \sum_{I,J,K} \bigl[\bigl(\Tilde{x}_{i_1} \cdots \Tilde{x}_{i_{|I|}}\bigr) \rhd r_1, \bigl(\Tilde{x}_{j_1} \cdots \Tilde{x}_{j_{|J|}}\bigr) \rhd r_2, \bigl(\Tilde{x}_{k_1} \cdots \Tilde{x}_{k_{|K|}}\bigr) \rhd r_3\bigr] \biggr) \\
 \qquad {}+ \sum_{I,J,K} \bigl( x_0 \rhd \bigl[\bigl(x_{i_1} \cdots x_{i_{|I|}}\bigr) \rhd r_1, \bigl(x_{j_1} \cdots x_{j_{|J|}}\bigr) \rhd r_2, \bigl(x_{k_1} \cdots x_{k_{|K|}}\bigr) \rhd r_3\bigr] \bigr)\\
 \qquad {}- \sum_{s=1}^k \sum_{I,J,K} \bigl[\bigl(\Tilde{x}_{i_1} \cdots \Tilde{x}_{i_{|I|}}\bigr) \rhd r_1, \bigl(\Tilde{x}_{j_1} \cdots \Tilde{x}_{j_{|J|}}\bigr) \rhd r_2, \bigl(\Tilde{x}_{k_1} \cdots \Tilde{x}_{k_{|K|}}\bigr) \rhd r_3\bigr] \\
 \qquad {}- \sum_{\hat{I},\hat{J},\hat{K}} \bigl[\bigl(x_{i_1} \cdots x_{i_{|\hat{I}|}}\bigr) \rhd r_1, \bigl(x_{j_1} \cdots x_{j_{|\hat{J}|}}\bigr) \rhd r_2, \bigl(x_{k_1} \cdots x_{k_{|\hat{K}|}}\bigr) \rhd r_3\bigr] \\
 \quad {}= x_0 \rhd ( (x_1 \cdots x_k) \rhd [r_1, r_2, r_3]) - \sum_{s=1}^k (x_1 \cdots (x_0 \rhd x_s) \cdots x_k) \rhd [r_1, r_2, r_3] \\
 \qquad {}- \sum_{\hat{I},\hat{J},\hat{K}} \bigl[\bigl(x_{i_1} \cdots x_{i_{|\hat{I}|}}\bigr) \rhd r_1, \bigl(x_{j_1} \cdots x_{j_{|\hat{J}|}}\bigr) \rhd r_2, \bigl(x_{k_1} \cdots x_{k_{|\hat{K}|}}\bigr) \rhd r_3\bigr] \\
 \quad {}= (x_0x_1 \cdots x_k) \rhd [r_1, r_2, r_3] \\
 \qquad {}- \sum_{\hat{I},\hat{J},\hat{K}} \bigl[\bigr(x_{i_1} \cdots x_{i_{|\hat{I}|}}\bigr) \rhd r_1, \bigl(x_{j_1} \cdots x_{j_{|\hat{J}|}}\bigr) \rhd r_2, \bigl(x_{k_1} \cdots x_{k_{|\hat{K}|}}\bigr) \rhd r_3\bigr] \in \mathcal{I}.
\end{gather*}
Hence the statement is true for $n=k+1$, and the lemma follows by induction.
\end{proof}

\subsection*{Acknowledgements}
The authors are supported by the Research Council of Norway through project 302831 Computational Dynamics and Stochastics on Manifolds (CODYSMA). Thanks to Kristoffer F{\o}llesdal for early discussions on this research.

\pdfbookmark[1]{References}{ref}
\LastPageEnding

\end{document}